\documentclass[xcolor,table,nobib]{tufte-handout}

\usepackage{matyi}
\usepackage{tkz-euclide}
\usepackage[oldstyle,p]{fbb}
\usepackage[backend=biber,autocite=footnote,style=authoryear]{biblatex}
\usepackage[left]{lineno}

\newcommand{\fe}{\text{ for every }}

\newcommand{\none}[1]{\norm*{#1}_{\mathbf 1}}
\newcommand{\noner}[1]{\norm*{#1}_{\mathbf 1,R}}
\newcommand{\nmean}[1]{\overbar{\M}\pa{#1}}
\newcommand{\nmeanr}[1]{\overbar{\M}_R\pa{#1}}
\newcommand{\ninfty}[1]{\norm*{#1}_{\mathbf \infty}}
\newcommand{\nvar}[1]{\norm*{#1}_{\mathbf V}}
\newcommand{\rr}{\raggedright}

\begin{document}
\title {Generation of measures on the torus with good sequences of
  integers}

\author{E. Lesigne, A. Quas, J. Rosenblatt, M. Wierdl}


\maketitle
\begin{abstract}
  Let $S\coloneqq (s_1<s_2<\dots)$ be a strictly increasing sequence
  of positive integers and denote $\e(\beta)\cq e^{2\pi i \beta}$. We
  say $S$ is \emph{good} if for every real $\alpha$ the limit
  $\lim_N \frac1N\sum_{n\le N} \e(s_n\alpha)$ exists. By the Riesz
  representation theorem, a sequence $S$ is good iff for every real
  $\alpha$ the sequence $\pa*{s_n\alpha}$ possesses an asymptotic
  distribution modulo 1.  Another characterization of a good sequence
  follows from the spectral theorem: the sequence $S$ is good iff in
  any probability measure preserving system $(X,\m,T)$ the limit
  $\lim_N \frac1N\sum_{n\le N}f\pa{T^{s_n}x}$ exists in $L^2$-norm for
  $f\in L^2(X)$.

  Of these three characterization of a good set, the one about limit
  measures is the most suitable for us, and we are interested in
  finding out what the limit measure
  $\mu_{S,\alpha}\cq \lim_N\frac1N\sum_{n\le N} \delta_{s_n\alpha}$ on
  the torus can be. In this first paper on the subject, we investigate
  the case of a single irrational $\alpha$.  We show that if $S$ is a
  good set then for every irrational $\alpha$ the limit measure
  $\mu_{S,\alpha}$ must be a continuous Borel probability measure.
  Using random methods, we show that the limit measure
  $\mu_{S,\alpha}$ can be any measure which is absolutely continuous
  with respect to the Haar-Lebesgue probability measure on the torus.
  On the other hand, if $\nu$ is the uniform probability measure
  supported on the Cantor set, there are some irrational $\alpha$ so
  that for no good sequence $S$ can we have the limit measure
  $\mu_{S,\alpha}$ equal $\nu$.  We leave open the question whether
  for any continuous Borel probability measure $\nu$ on the torus
  there is an irrational $\alpha$ and a good sequence $S$ so that
  $\mu_{S,\alpha}=\nu$.
\end{abstract}
\tableofcontents

\section{Introduction, main results}
\label{sec:intr-main-results}

Throughout the paper we will use the \emph{arithmetic average operator
  $\setA$}: for a finite index set $S$, a vector space $V$ and a
$S\to V$ function $f$ we define $\setA_Sf(s)$
\begin{equation}
  \label{eq:1}
  \setA_Sf(s)=\setA_{s\in S}f(s)\cq \frac1{\# S}\sum_{s\in S}f(s)
\end{equation}
where $\#S$ denotes the number of elements in $S$.

We use the convention that if an interval appears as an index set in a
summation then we consider only the integers in the interval.  For
example, $\sum_{n\in [0,N)}a_n=\sum_{n\in\{0,1,\dots,N-1\}}a_n$.

We also use Weyl's notation $\e(\beta)\cq e^{2\pi i\beta}$.  Note that
$\e^p(\beta)=\e(p\beta)$ for every integer $p$.

We denote by $\setT$ the torus $\setR/\setZ$ and we represent it as
the unit closed interval $[0,1]$ with $0=1$.

\subsection{Good sequences, main question}
\label{sec:good-sequences-main}

\begin{defn}[label={defn:1}]{Good sequence}{}\rr
  We say that a sequence $S=(s_n)_{n\in\setN}$ of integers is
  \emph{good} if the limit $\lim_N \setA_{n\in[1,N]}\e\pa*{s_n\alpha}$
  exists for every real number $\alpha$.
\end{defn}
Good sequences have been extensively studied in many parts of
mathematics, such as in number theory and ergodic theory.

In this paper we restrict our attention to strictly increasing
sequences $S$ of positive integers in which case we can and will
consider $S$ as a subset of $\setN$, and we'll use the concept of good
sequence and good set interchangeably.

Among the wellknown good sequences are the full set $\setN$ of
positive integers\autocite{MR1511862} ,  the sequence
$(n^2)_{n\in\setN}$ of squares\autocite{MR1511862} and the sequence
$(p_n)_{n\in\setN}$ of primes\autocite{zbMATH03026053} where $p_n$
denotes the $n$th prime number. For these sequences the limits
$\lim_N\setA_{n\in[1,N]}\e(s_n\alpha)$ are as follows

\begin{equation}
  \label{eq:2}
  \begin{split}
    \lim_N\setA_{n\in[0,N)}\e(n\alpha)&=
                                        \begin{dcases*}
                                          1& if $\alpha=1$\\
                                          0& if $\alpha\ne 0$
                                        \end{dcases*}
    \\
    \lim_N\setA_{n\in[0,N)}\e(n^2\alpha)&=
                                          \begin{dcases*}
                                            \setA_{b\in[1,q]}\e\pa*{b^2\frac{a}{q}}
                                            & if $\alpha=\frac{a}{q}$, $\gcd(a,q)=1$ \\
                                            0 & if $\alpha$ is irrational
                                          \end{dcases*}
    \\
    \lim_N\setA_{n\in[1,N]}\e(p_n\alpha)&=
                                          \begin{dcases*}
                                            \setA_{\substack{b\in[1,q]\\\gcd(b,q)=1}}\e(b/q)
                                            & if $\alpha=\frac{a}{q}$, $\gcd(a,q)=1$ \\
                                            0 & if $\alpha$ is irrational
                                          \end{dcases*}
  \end{split}
\end{equation}

In case of a good sequence $S=(s_n)$ and a fixed $\alpha$, the
existence of $\lim_N\setA_{n\in[1,N]}\e(s_np\alpha)$ for every
$p\in\setZ$ implies, by uniform approximation of a continuous
$\setT\to\setC$ function by trigonometric polynomials, that for every
continuous $\setT\to\setC$ function $\phi$ the limit
$\lim_N\setA_{n\in[1,N]}\phi(s_n\alpha)$ exists. By the Riesz
representation theorem, this implies that the weak limit
$\lim_N\setA_{n\in[1,N]}\delta_{s_n\alpha}$ of discrete measures
$\setA_{n\in[1,N]}\delta_{s_n\alpha}$ on $\setT$ exists.

By this argument, the existence of
$\lim_N\setA_{n\in[1,N]}\e(s_n\alpha)$ for every $\alpha$ implies the
existence of the limit measure
$\lim_N\setA_{n\in[1,N]}\delta_{s_n\alpha}$ for every $\alpha$.
Denote the Haar-Lebesgue probability measure on the torus $\setT$ by
$\lambda$ and recall that the Fourier coefficients $\lambda(\e^p)$ of
$\lambda$ satisfy
\begin{equation}
  \label{eq:3}
  \lambda(\e^p)=
  \begin{dcases*}
    1                                                                        & for $p=0$\\
    0                                                                        & for $p\in\setZ$, $p\ne 0$
  \end{dcases*}
\end{equation}
where for a given measure $\nu$ and $\nu$-integrable function $\phi$,
we use\footnote{and will use troughout the paper} the functional
notation $\nu(\phi)$ for the integral of $\phi$ with respect to $\nu$,
\begin{equation}
  \label{eq:4}
  \nu(\phi)=\int \phi\di \nu
\end{equation}
For our three good sets the limit measures are as follows.
\begin{equation}
  \label{eq:5}
  \begin{split}
    \lim_N\setA_{n\in[1,N]}\delta_{n\alpha}                                  & =
                                                                               \begin{dcases*}
                                                                                 \setA_{b\in[1,q]}\delta_{b/q} & if $\alpha=a/q$, $\gcd(a,q)=1$         \\
                                                                                 \lambda                       & if $\alpha$ is irrational
                                                                               \end{dcases*}
    \\
    \lim_N\setA_{n\in[1,N]}\delta_{n^2\alpha}                                & =
                                                                               \begin{dcases*}
                                                                                 \setA_{b\in[1,q]}\delta_{b^2\frac{a}{q}}
                                                                                 & if $\alpha=\frac{a}{q}$, $\gcd(a,q)=1$ \\
                                                                                 \lambda                     & if $\alpha$ is irrational
                                                                               \end{dcases*}
    \\
    \lim_N\setA_{n\in[1,N]}\delta_{p_n\alpha}                                & =
                                                                               \begin{dcases*}
                                                                                 \setA_{\substack{b\in[1,q]                                           \\\gcd(b,q)=1}}\delta_{b/q}
                                                                                 & if $\alpha=\frac{a}{q}$, $\gcd(a,q)=1$ \\
                                                                                 \lambda                     & if $\alpha$ is irrational
                                                                               \end{dcases*}
  \end{split}
\end{equation}
What we see in these three examples is that in case of irrational
$\alpha$ the limit measure is the Haar-Lebesgue measure $\lambda$ and
in case of rational $\alpha=a/q$, $\gcd(a,q)=1$, the limit measure is
supported on a subset of the $q$th roots of unity and appears to be
quite uniform on its support.  In case of irrational $\alpha$, the
simplest question is if it's possible that the limit measure is not
$\lambda$.  In case of rational $\alpha$, we can ask if the limit
measure always has to show some kind of uniformity.

Let us consider a good sequence $S=(s_n)$. The existence of the limit
$\lim_N\setA_{n\in[1,N]}\e(s_n\alpha)$ for every $\alpha$ implies that
the weak limit $\lim_N\setA_{n\in[1,N]}\delta_{s_n\alpha}$ of discrete
measures $\setA_{n\in[1,N]}\delta_{s_n\alpha}$ on $\setT$ exists for
every $\alpha$.  Let us denote this weak limit measure by
$\mu_{S,\alpha}$,
\begin{equation}
  \label{eq:6}
  \mu_{S,\alpha}\cq \lim_N \setA_{n\in[1,N]}\delta_{s_n\alpha}
\end{equation}
The main question we want to investigate in this paper is
\begin{quest}[label={quest:1}]{Main question}{}
  What can the limit measure $\mu_{S,\alpha}$ be? Can it be any Borel
  probability measure on $\setT$?
\end{quest}

\subsection{Main results}
\label{sec:main-results}

As we stated earlier, we try to answer \cref{quest:1} for strictly
increasing sequences, and unless we say otherwise, we assume from now
on that $S=(s_n)$ is a strictly increasing sequence of positive
integers which we often consider as a subset of $\setN$.

Our first observation is that the answer to \cref{quest:1} will depend
on $\alpha$.  If $\alpha$ is a rational number, say,
$\alpha=\frac{a}{q}$ with $\gcd(a,q)=1$, then the limit measure is
clearly supported on the set
\begin{equation}
  \label{eq:7}
  \setT_q\cq\set{b/q}{b\in[1,q]}
\end{equation}
of $q$th roots of unity.  So the question is if the limit measure
$\mu_{S,a/q}$ can be any probability measure supported on $\setT_q$?
The answer is yes.  First a terminology.
\begin{defn}[label={defn:2}]{Representable measure at $\alpha$}{}\rr
  Let $S$ be a good set, and let $\nu$ be a nonzero, finite Borel
  measure on $\setT$.

  We say that \emph{$S$ represents $\nu$ at $\alpha\in\setT$} if
  $\mu_{S,\alpha}=\frac1{\nu(\setT)}\nu$.

  We say \emph{$\nu$ is representable at $\alpha$} if there is a good
  set which represents $\nu$ at $\alpha$.
\end{defn}

\begin{thm}[label={thm:representation_rational}]{Every probability
    measure on $\setT_q$ can be represented}{}\rr 
  Let $q$ and $a$ be positive integers with $\gcd(a,q)=1$, and let
  $\nu$ be a probability measure supported on the set $\setT_q$ of
  $q$th roots of unity.

  Then $\nu$ can be represented at $\frac{a}{q}$, that is, there is a
  good set $S$ so that $\mu_{S,\frac{a}{q}}=\nu$.
\end{thm}
Before discussing the limit measure $\mu_{S,\alpha}$ for irrational
$\alpha$, let us note the following fact which will help us appreciate
the concept of a good set.

Suppose we are given an irrational number $\alpha\in\setT$ and a Borel
probability measure $\nu$ on $\setT$. We claim that there exists a
sequence $(x_n)$ in $\setT$ with asymptotic distribution $\nu$,
i. e. such that $\lim_N\setA_{n\in[1,N]}\delta_{x_n}=\nu$. Considering
such a sequence and using the density of the sequence $(n\alpha)_n$ in
$\setT$, we can select a strictly increasing sequence $(s_n)$ of
integers so that $\lim_n(s_n\alpha-x_n)=0\mod 1$, and we have
$\lim_N\setA_{n\in[1,N]}\delta_{s_n\alpha}=\nu$. Taking
$S=\set{s_n}{n\in\setN}$, we could say that $\mu_{S,\alpha}=\nu$, but
nothing insures us that the set $S$ is good.

There are different ways to prove the preceding claim. For example we
can pick the numbers $x_n$ randomly and independently with law $\nu$,
and the strong law of large numbers asserts that the sequence $(x_n)$
has, almost surely, the right asymptotic distribution.

It is particularly simple to get a point-mass as a limit measure.  For
example, to get the Dirac measure at $1/2$, so $\nu=\delta_{1/2}$,
take a strictly increasing sequence $(s_n)$ of natural numbers so that
$s_n\alpha$ converges to $1/2$ $\mod 1$, and let
$S\cq\set*{s_n}{n\in\setN}$.  In contrast to this example, for good
sets we have a dramatic departure from the case of rational $\alpha$.

\begin{thm}[label={thm:cont_only}]{$\mu_{S,\alpha}$ is continuous for
    irrational $\alpha$}{}\rr
  Only continuous measures can be represented at an irrational number.

  To spell this out, let $S=(s_n)$ be a good sequence and $\alpha$
  be an irrational number.

  Then the limit Borel probability measure
  $\mu_{S,\alpha} =\lim_N\setA_{n\in [1,N]}\delta_{s_n\alpha}$ is a
  continuous measure.

\end{thm}
The obvious question in turn is if any given continuous Borel
probability measure can be represented at any irrational number. The
answer is no, as the next result shows.
\begin{thm}[label={cantor_nonrep}]{Some continuous measures cannot be
    represented at every irrational point}{}\rr
  Let $\nu$ be a Borel probability measure on $\setT$ so that its
  Fourier coefficients do not converge to $0$, so
  \begin{equation}
    \label{eq:8}
    \limsup_{p\to\infty}\abs*{\mu\pa*{\e^p}}>0
  \end{equation}

  Then there is a set $A\subset \setT$ of full Lebesgue measure so
  that $\nu$ cannot be represented at any $\alpha\in A$.
\end{thm}
Since a measure $\nu$ is called a \emph{Rajchman}
measure\autocite{MR1364897} if its Fourier coefficients vanish at
infinity, that is, $\lim_p\nu\pa*{\e^p}=0$, we can rephrase
\cref{cantor_nonrep} by saying that if $\nu$ is representable at every
irrational $\alpha$ then it must be a Rajchman measure.  A well known
non-Rajchman continuous measure is the uniform measure on the triadic
Cantor set.

While \cref{cantor_nonrep} doesn't exclude the possibility that
$A=\setT$, that is, a non-Rajchman measure cannot be represented
anywhere, Christophe Cuny and François
Parreau\autocite{parreau:hal-03805242} constructed a non-Rajchman
measure which is representable at uncountably many $\alpha$'s.
Nevertheless, the following question remains open.
\begin{quest}[label={quest:2}]{Is every continuous measure
    representable somewhere?}{}
  Let $\nu$ be a continuous Borel probability measure on $\setT$.

  Is there an irrational $\alpha$ so that $\nu$ is representable at
  $\alpha$?
\end{quest}
The next result says that if $\nu$ is absolutely continuous with
respect to the Lebesgue probability measure $\lambda$ on the torus
$\setT$, then it can be represented at every irrational $\alpha$.

\begin{thm}[label={thm:abs_cont_reprable}]{Absolutely continuous
    measures are representable at every irrational point}{}\rr 
  Let $\nu$ be a Borel probability measure on $\setT$ which is
  absolutely continuous with respect to the Lebesgue probability
  measure on $\setT$.  Let $\alpha$ be an irrational number.

  Then $\nu$ is representable at $\alpha$.
\end{thm}
Our proof of \cref{thm:abs_cont_reprable} is flexible and enables us
to show a more general result, namely it turns out that a given
absolutely continuous measure can be represented by a good subset of
any given good set, provided it doesn't increase too fast, it is
sublacunary.  For a given set $R\subset \setN$ let $R(N)$ denote the $N$th initial
segment of $R$,
\begin{equation}
  \label{eq:9}
  R(N)\cq R\cap[1,N]
\end{equation}
We say $R$ is \emph{sublacunary}\footnote{Traditionally, $(r_n)$ is
  called lacunary if it satisfies $\liminf_n\frac{r_{n+1}}{r_n}>1$,
  and such a sequence satisfies $\#R(N)=O(\log N)$.  Traditionally, a
  sublacunary sequence is one that satisfies
  $\lim_n\frac{r_{n+1}}{r_n}=1$ and such a sequence satisfies
  $\lim_N\frac{\#R(N)}{\log N}=\infty$.  Our definion of a sublacunary
  sequence in \cref{eq:10} describes sequences which satisfy
  $\liminf_n\frac{r_{n+1}}{r_n}=1$ but may not satisfy
  $\lim_n\frac{r_{n+1}}{r_n}=1$.} if it satisfies the growth condition
\begin{equation}
  \label{eq:10}  
  \lim_N\frac{\# R(N)}{\log N}=\infty
\end{equation}
In case we consider the sequence $(r_n)$ instead of the set $R$, it's
more useful to write \cref{eq:10} in the form
\begin{equation}
  \label{eq:11}
  \lim_N\frac{N}{\log r_N}=\infty
\end{equation}

\begin{thm}[label={thm:general_representability}]{Absolutely
    continuous measures can be represented by subsets of a good
    set}{}\rr
  Let $R$ be a sublacunary good set.  Let $\alpha$ be an irrational
  number, and let the Borel probability measure $\nu$ be absolutely
  continuous with respect to $\mu_{R,\alpha}$.

  Then there is a good set $S\subset R$ which represents $\nu$ at
  $\alpha$.
\end{thm}
\marginnote{As a consequence of ~\cref{thm:general_representability},
  every measure which is absolutely continuous with respect to the
  Lebesgue measure can be represented at any given irrational $\alpha$
  by a subset of the primes, squares, or
  $\set*{\intpart*{n^2\log n}}{n\in\setN}$.} We will see that the
proof of \cref{thm:general_representability} reveals a close
connection between the Radon-Nikodym derivative $\rho$ of $\nu$ with
respect to $\mu_{R,\alpha}$ and the relative mean\footnote{The usual
  terminology is relative \emph{density} instead of relative mean, but
  we will use the more general concept of the \emph{mean of a
    $R\to\setC$ function} in \cref{sec:weighted-averages} and we
  prefer to use a single terminology and notation for economical
  reasons.} of the set $S$ representing $\nu$.  For a given
$R\subset\setN$ and $S\subset R$, the \emph{relative mean $\M_R(S)$}
of $S$ in $R$ is defined by
\begin{equation}
  \label{eq:12}
  \M_R(S)\cq \lim_N\frac{\# S(N)}{\# R(N)}
\end{equation}
provided the limit on the right exists.  The \emph{relative upper mean
  $\overbar{\M}_R(S)$} of $S$ in $R$ is defined by
\begin{equation}
  \label{eq:13}
  \overbar{\M}_R(S)\cq \limsup_N\frac{\# S(N)}{\# R(N)}
\end{equation}
In case $R=\setN$, we suppress the base set in our notation, and we
write $\M(S)$ for $\M_{\setN}(S)$ and $\overbar{\M}(S)$ for
$\overbar{\M}_{\setN}(S)$.
\begin{thm}[label={thm:positive_density_rho}]{Connection between
    $\frac{\di \nu }{\di \mu_{R,\alpha}}$, $\M_R(S)$ and
    $\overbar{\M}_R(S)$}{}\rr
  Let $R$ be a sublacunary good set.
  \begin{thmenum}
  \item \label{item:1} For an irrational $\alpha$ let the unsigned
    function $\rho\in L^1\pa*{\mu_{R,\alpha}}$ with
    $\mu_{R,\alpha}(\rho)=1$ be bounded so
    $\norm{\rho}_{L^{\infty}\pa*{\mu_{R,\alpha}}}<\infty$.

    Then there is a good set $S\subset R$ representing the measure
    $\rho\cdot \mu_{R,\alpha}$ at $\alpha$ and satisfying
    $\M_R(S)=\frac1{\norm{\rho}_{L^{\infty}\pa*{\mu_{R,\alpha}}}}$.
  \item \label{item:2} Let $S$ be a good subset of $R$ with positive
    upper density in $R$, so $\overbar{\M}_R(S)>0$.

    Then for every irrational $\beta$ the limit measure
    $\mu_{S,\beta}$ is absolutely continuous with respect to
    $\mu_{R,\beta}$.  Furthermore, the Radon-Nikodym derivative
    $\rho_\beta\cq \frac{\di \mu_{S,\beta}}{\di \mu_{R,\beta}}$ is a
    bounded function satisfying
    $\norm{\rho_\beta}_{L^{\infty}\pa*{\mu_{R,\beta}}}\le
    \frac1{\overbar{\M}_R(S)}$.
  \end{thmenum}
\end{thm}
We see that \cref{thm:general_representability} gives a \emph{full
  characterization} of the limit measure for sets with positive upper
mean\footnote{so now $R=\setN$}, giving an exact relationship between
the upper mean of the set and the bound of the RN derivative: On the
one hand if $\overbar{\M}(S)>0$, the limit measure $\mu_{S,\beta}$ for
every $\beta$ must be absolutely continuous with respect to $\lambda$
with bounded RN derivative $\rho_\beta$ satisfying
$\norm{\rho_\beta}_{L^{\infty}(\lambda)}\le \frac1{\overbar{\M}(S)}$.
On the other hand, any Borel probability measure $\nu$ which is
absolutely continuous with respect to $\lambda$ with bounded, nonzero
RN derivative $\rho$ is representable at any irrational $\alpha$ with
a set of positive mean satisfying
$\M(S)=\frac1{\norm{\rho}_{L^{\infty}\pa*{\lambda}}}$.

\Cref{item:2} has the following consequence.
\begin{cor}[label={conj:1}]{If the RN derivative $\rho$ is unbounded,
    then $\M_R(S)=0$ }{}\rr
  Let $R$ be a good set and $\alpha$ an irrational number. Suppose the
  unsigned function $\rho\in L^1\pa*{\mu_{R,\alpha}}$ with
  $\mu_{R,\alpha}(\rho)=1$ is unbounded, and that the good set
  $S\subset R$ represents the measure $\rho\cdot \mu_{R,\alpha}$ at
  $\alpha$.
  
  Then $S$ must have $0$ mean in $R$, so $\M_R(S)=0$.
\end{cor}

\subsection{Weighted averages}
\label{sec:weighted-averages}

Our results in
\cref{thm:abs_cont_reprable,thm:general_representability} will be
consequences, via a random procedure, of results on weighted averages.

We need to fix some terminology and notation.  We define the
Besicovitch type seminorm $\none{ }$ for all complex valued sequences
$f\in\setC^{\setN}$ by
\begin{equation}
  \label{eq:14}
  \none{f}\cq \limsup_N\setA_{[1,N]}|f|, \quad f\in\setC^{\setN}
\end{equation}
The number $\mathbf 1$ in the subscript of $\none{}$ expresses the
similarity of this norm to the $L^1$ norm.

For a set $S\subset \setN$, we may use the notation $\none{S}$ instead
of $\none{\setone_S}$, though in this case we do not get a new
concept, since $\none{S}=\overbar{\M}(S)$.

For an infinite set $R\subset \setN$ we define the relative
$\mathbf 1$-norm $\noner{f }$ of a complex valued $R\to\setC$ function
by
\begin{equation}
  \label{eq:15}
  \noner{f}\cq \limsup_N\setA_{R(N)}|f|, \quad f\in\setC^{R}
\end{equation}
If the set $R$ is given as a strictly increasing sequence $(r_n)$ and
for an $f\in\setC^{R}$ we define $F$ by $F(n)\cq f(r_n)$, then
$\noner{f}=\none{F}$.

Let $R\subset \setN$ be an infinite set. The $R\to \setR$ function $w$
is called a \emph{$R$-weight} if $w$ is unsigned, so $w\ge 0$, and
$\sum_{r\in R}w(r)=\infty$. We may refer to an $R$-weight as ``a
weight supported on $R$''.

An $R$-weight $w$ can be considered a measure on the set $R$ and in
that case for $S\subset R$ we may briefly write $w(S)$ for the sum
$\sum_{s\in S}w(s)$.

For a finite set $S\subset \setN$ let $\sigma$ be a real valued,
unsigned function defined on $S$.  We can consider $\sigma$ a measure
on $S$, and as such, we assume $\sigma(S)>0$. For a vector space $V$
and $S\to V$ function $f$, define the $\sigma$-weighted average
$\setA^\sigma_Sf$ of $f$ on $S$ by
\begin{equation}
  \label{eq:16}
  \setA^\sigma_{S}f=\setA^\sigma_{s\in S}f(s)\cq
  \frac1{\sigma(S)}\sum_{s\in S}\sigma(s)f(s)
\end{equation}
\begin{defn}[label={defn:3}]{Good weights and represented measures by
    them}{}\rr
  Let $R\subset \setN$ be infinite.  Let $w$ be an $R$-weight.

  We say $w$ is a \emph{good $R$-weight} if the weak limit
  $\lim_N \setA^w_{r\in R(N)}\delta_{r\beta}$ exists for every
  $\beta\in\setT$.  We denote this limit by $\mu_{w,\beta}$,
  \begin{equation}
    \label{eq:17}
    \mu_{w,\beta}\cq \lim_N\setA^w_{r\in R(N)}\delta_{r\beta} 
  \end{equation}
  Let $\nu$ be a Borel probability measure on $\setT$ and let
  $\alpha\in\setT$.

  We say the \emph{$R$-weight $w$ represents $\nu$ at $\alpha$} if $w$
  is good and $\mu_{w,\alpha}=\nu$.
  
\end{defn}
\marginnote{Note the following form of the definition of the limit
  measure $\mu_{w,\alpha}$ when  we consider $R$ as the strictly
  increasing sequence $(r_n)$:
  $\mu_{w,\alpha}=\lim_N\setA^w_{n\in [1,N]}\delta_{r_n\beta}$, so now
  we have
  $\setA^w_{n\in
    [1,N]}\delta_{r_n\beta}=\frac1{\sum_{n\in[1,N]}w(r_n)}\sum_{n\in[1,N]}w(r_n)\delta_{r_n\alpha}$.}

Note the following characterization of good weights: The $R$-weight
$w$ is good iff the limit $\lim_N\setA^w_{r\in R(N)}\e(r\alpha)$
exists for every $\alpha$.

In the special case of a good set $S\subset \setN$, we have
$\mu_{S,\alpha}=\mu_{\setone_S,\alpha}$ since the weighted averages
with weight $w\cq \setone_S$ correspond to the averages along $S$.

In contrast to good sets, the representation of absolutely continuous
measures by weights can always be accomplished by weights with
positive, finite mean. In fact, the representing weight has an
additional property.
\begin{defn}[label={defn:integrable_weight}]{Integrable weight}{}\rr
  Let $R\subset \setN$ be infinite.

  We call the $R$-weight $w$ \emph{integrable} if it can be
  approximated arbitrary closely in the seminorm $\noner{}$ by
  bounded, good weights: for every $\epsilon>0$ there is a good
  $R$-weight $v$ with $\ninfty{v}<\infty$ so that
  $\noner{v-w}<\epsilon$.
\end{defn}

\begin{thm}[label={thm:representation_by_weights}]{Representation by
    weights}{}\rr 
  Let $R$ be a good set.

  \begin{thmenum}
  \item \label{item:3} For an irrational $\alpha$ let the unsigned
    function $\rho\in L^1\pa*{\mu_{R,\alpha}}$ satisfy
    $\mu_{R,\alpha}(\rho)=1$.

    Then there is an integrable $R$-weight $w$ with $\M_R(w)=1$ which
    represents the measure $\rho\cdot \mu_{R,\alpha}$ at $\alpha$. If
    $\rho\in L^\infty\pa*{\mu_{R,\alpha}}$ then the $R$-weight $w$
    representing the measure $\rho\cdot \mu_{R,\alpha}$ can also
    satisfy $\norm*{\rho}_{L^\infty\pa*{\mu_{R,\alpha}}}=\ninfty{w}$.

  \item \label{item:4} Let $w$ be a good, integrable $R$-weight which
    satisfies $\noner{w}>0$.

    Then for every $\beta$ the limit measure $\mu_{w,\beta}$ is
    absolutely continuous with respect to $\mu_{R,\beta}$.
  \end{thmenum}
\end{thm}

\subsection{Applications in ergodic theory}
\label{sec:appl-ergod-theory}

Besides the intrinsic interest of our main question, \cref{quest:1},
there may be several applications of studying limit measures.  One
major application is in ergodic theory.

Recall that a measure preserving dynamical system is a probability
space $(X,\m)$, where $\m(X)=1$, equipped with a measurable, measure
preserving transformation $T$ of $X$.  By the spectral theorem, a good
set has the following characterization: the sequence $S=(s_n)$ of
positive integers is good iff the limit
$\lim_N\setA_{n\in[1,N]}f(T^{s_n}x)$ exists in $L^2(X)$-norm in any
measure preserving dynamical system $(X,\m,T)$ for any $f\in L^2(X)$.

This means that our work in describing the possible limit measures in
case of a good set yields an identification of the limit in mean
ergodic theorems.  Identification of the limit is often the crucial
step in some applications, and here we just mention two of these,
recurrence and almost sure convergence.  In case of studying
recurrence, the identification of the limit readily tells us whether a
given set is a set of recurrence.  In case of trying to see if some
ergodic averages converge almost everywhere, after the identification
of the $L^2$-limit, we usually want to see if there is some kind of
rate with which the averages converge to the $L^2$-limit.  For
example, this is the case when one proves that the ergodic averages
along the squares converge almost surely.  The application of the
circle method here is exactly a quantitative expression of how the
averages converge in $L^2$-norm.

\subsection{Future work}
\label{sec:future-work}

The techniques developed in this paper allow one to address the
\emph{simultaneous} representability of probability measures at
several different points of the torus, and we plan to explore this in
a future work. But which family $\set*{\nu_{\alpha}}{\alpha\in\setT}$
of measures can be represented by a single good set remains open even
if we restrict the family to absolutely continuous measures with
respect to the Lebesgue probability measure $\lambda$.  What we can
say at this point is that for a given good set $S$, the set of
$\alpha\in\setT$ where the limit measure $\mu_{S,\alpha}$ is not the
Lebesgue measure is small: it is both of first Baire category and of
$0$ measure under every Rajchman measure\autocites[Theorem
3]{MR799255}[see also][]{MR1364897} on $\setT$, so
$\nu\set{\alpha}{\mu_{S,\alpha}\ne \lambda}=0$ for every Rajchman
measure $\nu$.

\subsection{Summary of notation}
\label{sec:summary-notation}

We realize that we use quite extensive notation, many of which are
new, so we give a summary of our notations in \cref{table:notations}.

\begin{table}[h!]
  \rowcolors{3}{gray!30}{white}
  \caption{Notations}
  \bigskip
  \begin{tabular}{llll}
    \toprule
    Symbol           & Definition                                     & Parameters                              & Name                             \\
    \midrule
  $\setN$            & $\{1,2,3,\dots\}$                              &                                         & Natural numbers                  \\    
  $\setT$            &                                                &                                         & torus                            \\
    $\lambda$        &                                                &                                         & Haar-Lebesgue measure on $\setT$ \\
    $\e(\theta)$     & $\exp(2\pi i \theta)$                          & $\theta\in \setT$                       &                                  \\
    $\e^p(\theta)$   & $\e(p\theta)$                                  & $p\in \setZ$                            &                                  \\
    $S(N)$           & $S\cap[1,N]$                                   & $S\subset\setN$                         & initial segment of $S$           \\
    $\#S(N)$         & $\sum_{s\in S(N)}1$                            & $S\subset\setN$                         & counting function of $S$         \\
    $\setA_{S}f$     & $\frac1{\#S}\sum_{s\in S}f(s)$                 & set $S$ is finite                       & average of $f$ on $S$            \\
    $\setA^w_{S}f$   & $\frac1{w(S)}\sum_{s\in S}w(s)f(s)$
                     & $w$ is a weight on $S$                         & $w$-average of $f$ on set $S$                                              \\
    $\mu_{S,\alpha}$ & $\lim_N \setA_{s\in S(N)}\delta_{s\alpha}$     & $S\subset\setN$, $\alpha\in\setT$       & limit measure of $S$ at $\alpha$ \\
    $\mu_{w,\alpha}$ & $\lim_N \setA^w_{s\in S(N)}\delta_{s\alpha}$ & weight $w$ on $S$, $\alpha\in\setT$ & limit measure of $w$ at $\alpha$ \\
    $\nu(\phi)$      & $\int_\setT \phi\di\nu$ &  &                                  \\
    $\M(f)$          & $\lim_N\setA_{[1,N]}f$   & $f\in\setC^{\setN}$                     & mean of $f$   \\
    $\M_R(f)$        & $\lim_N\setA_{R(N)}f$  & $R\subset\setN$, $f\in\setC^{R}$ & relative mean of $f$ \\
    $\mathcal{M}$   & $\set{f}{f\in\setC^{\setN}, \M(f)\text{ exists and is finite}}$ & & sequences with mean \\
  $\nmean{f}$  & $\limsup_N\abs*{\setA_{[1,N]}f}$& $f\in\setC^{\setN}$ & upper mean  \\
      $\nmeanr{f}$  & $\limsup_N\abs*{\setA_{R(N)}f}$& $R\subset \setN$, $f\in\setC^{R}$ & relative upper mean  \\
    $\none{f}$ & $\limsup_N\setA_{[1,N]}|f|$    & $f\in\setC^{\setN}$  & $\mathbf 1$-seminorm          \\
    $\noner{f}$  & $\limsup_N\setA_{R(N)}|f|$  & $R\subset \setN$, $f\in\setC^{R}$ & relative $\mathbf 1$-seminorm \\                           
    $\mathcal{C}_+$    & $\set*{\phi}{\phi:\setT\to[0,1], \text{ continuous }}$          &   & \\
    $\nvar{\nu_1-\nu_2}$ & $\sup_{\phi\in\mathcal{C}_+}\abs*{\nu_1(\phi)-\nu_2(\phi)}$     & $\nu_i$ finite Borel  measures  on $\setT$ & variation distance            \\
    \bottomrule
  \end{tabular}
  

  \medskip

  \label{table:notations}
\end{table}

\section{Basic example for representation}
\label{sec:basic-example-repr}

In this section we want to work out a rather simple but instructive
example, which will then motivate and form the basis of many of our
constructions later on.  When we are done with presenting this
example, we in fact proved \cref{thm:general_representability} in case
the Radon-Nikodym derivative is the indicator of a Jordan measurable
set.

Let $\alpha$ be irrational and let $I\subset\setT$ be an interval. We
want to show that if a probability measure $\nu$ is absolutely
continuous with respect to $\lambda$ with the Radon-Nikodym derivative
equal $\setone_I$, the indicator of $I$, then there is a set $S$ which
represents $\nu$ at $\alpha$. Probably the simplest way\footnote{We
  could also define such a set by taking
  $\set{n}{n\in\setN,n^2\alpha\in I\pmod 1}$ or
  $\set{p}{p\in\mathcal{P},p\alpha\in I\pmod 1}$ where $\mathcal{P}$
  is the set of primes.} to define such a set $S$ is by taking
\begin{equation}
  \label{eq:18}
  S=\set{n}{n\in\setN,n\alpha\in I}
\end{equation}
There are two things to verify.  First, that $S$ is indeed a good set,
and to do that, we need to show that the weak limit
$ \mu_{S,\beta}=\lim_N\setA_{s\in S(N)}\delta_{s\beta}$ exists for
every $\beta$.  Second, we then have to verify that
$ \mu_{S,\alpha}=\frac1{\lambda(I)}\setone_I\cdot \lambda$.  The
second one, in fact, is almost instantaneous to do since it follows
from the uniform distribution of $(n\alpha)_{n\in \setN}\pmod 1$. To
see how it follows, it's enough to show that for every interval
$J\subset\setT$ we have
$ \mu_{S,\alpha}(J)=\lambda\pa*{\setone_J\cdot
  \frac1{\lambda(I)}\setone_I}$, that is
\begin{equation}
  \lim_N\setA_{s\in S(N)}\setone_J(s\alpha)=\frac1{\lambda(I)}\lambda(J\cap I)
\end{equation}
The left hand side can be written as
\begin{align*}
  \lim_N\setA_{s\in S(N)}\setone_J(s\alpha)
  &=\lim_N\frac{N}{\#S(N)}\setA_{n\in
    [1,N]}\setone_I(n\alpha)\setone_J(n\alpha)\\
  \intertext{since $\lim_N\frac{\#S(N)}N=\lambda(I) $ by the uniform distribution of  $(n\alpha)_{n\in
  \setN}$ $\pmod 1$,}
  &=\frac1{\lambda(I)} \lim_N\setA_{n\in[1,N]}\setone_{I\cap J}(n\alpha)\\
  \intertext{again by the unifom distribution of $(n\alpha)_{n\in\setN}\pmod 1$}
  &=\frac1{\lambda(I)}\lambda(I\cap J)
\end{align*}
To show that the weak limit
$ \mu_{S,\beta}=\lim_N \setA_{s\in S(N)}\delta_{s\beta}$ exists for
every $\beta$, it's enough to show that
$\lim_N\setA_{s\in S(N)}\e(s\beta)$ exists for every $\beta$. Since
\begin{equation}
  \setA_{s\in S(N)}\e(s\beta)
  =\frac{N}{\#S(N)}\setA_{n\in[1,N]}\setone_I(n\alpha)\e(n\beta)
\end{equation}
and since $\lim_N\frac{\#S(N)}N=\lambda(I) $, it's enough to show that
the limit $\lim_N\setA_{n\in[1,N]}\setone_I(n\alpha)\e(n\beta)$ exists
for every $\beta\in\setT$. To see this, first note that if we replace
$\setone_I$ by the character $\e^k$ the limit of
$\setA_{n\in [1,N]}\e^k(n\alpha)\e(n\beta)=\setA_{n\in
  [1,N]}\e\paren*{n(k\alpha+\beta)}$ as $N\to\infty$ exists and is as
follows
\begin{equation}
  \label{eq:19}
  \lim_N\setA_{n\in [1,N]}\e^k(n\alpha)\e(n\beta)=
  \begin{cases*}
    1& if $\beta=-k\alpha\pmod 1$\\
    0&otherwise
  \end{cases*}
\end{equation}
From this we get that if we replace $\setone_I$ by a trigonometric
polynomial $\phi$, the limit of
$\setA_{n\in [1,N]}\phi(n\alpha)\e(n\beta)$ exists and can be given
explicitly as\footnote{Notice that in \cref{eq:20}
  $\lambda\pa*{\phi\e^{k}}$ is the $k$th Fourier coefficient of
  $\phi$.}
\begin{equation}
  \label{eq:20}
  \lim_N\setA_{n\in[1,N]}\phi(n\alpha)\e(n\beta)=
  \begin{cases*}
    \lambda\pa*{\phi\e^{k}}& if $\beta=-k\alpha\pmod 1$\\
    0&otherwise
  \end{cases*}
\end{equation}
Using Weierstrass' theorem on being able to uniformly approximate a
continuous function by trigonometric polynomials, we can verify that
in \cref{eq:20} we can take $\phi$ to be any continuous function.

Now, to go from continuous functions to the indicator $\setone_I$ of
any interval $I$, it is enough to know that the indicator $\setone_I$
can be sandwiched between two unsigned continuous functions whose
integrals (with respect to $\lambda$) are arbitrarily close.  We thus
have
\begin{equation}
  \label{eq:21}
  \lim_N\setA_{n\in[1,N]}\setone_I(n\alpha)\e(n\beta)=
  \begin{cases*}
    \lambda(\setone_I\e_{k})& if $\beta=-k\alpha\pmod 1$\\
    0&otherwise.
  \end{cases*}
\end{equation}
We finally get, since
$
\mu_{S,\beta}(\e)=\frac1{\lambda(I)}\lim_N\setA_{n\in[1,N]}\setone_I(n\alpha)\e(n\beta)$,
\begin{equation}
  \label{eq:22}
  \mu_{S,\beta}(\e)=
  \begin{cases*}
    \frac1{\lambda(I)}\lambda(\setone_I\e^{k})& if $\beta=-k\alpha\pmod 1$\\
    0&otherwise
  \end{cases*}
\end{equation}
The above shows that $ \mu_{S,\beta}(\e)$ can be nonzero only if
$\beta$ is an integer multiple of $\alpha$, and we recognize
$\lambda\paren*{\setone_I\e^{k}}$ as the $k$th Fourier coefficient of
the function $\setone_I$, that is,
$\frac1{\lambda(I)}\lambda\paren*{\setone_I\e^{k}}$ is the $k$th
Fourier coefficient of the measure
$\frac1{\lambda(I)}\setone_I\lambda$.

One can rather easily extend this example in two ways.  First, the
proof can be repeated almost verbatim for the case when we take any
Jordan measurable set $B$ in place of the interval $I$. Indeed, all we
need to remark is that a set $B$ is Jordan measurable iff, for every
given $\epsilon>0$, its indicator function $\setone_B$ can be
sandwiched between two unsigned, continuous functions $\phi_a$ and
$\phi_b$ so that $\phi_b\le \setone_B\le \phi_a$ and
$\lambda\pa*{\phi_a-\phi_b}<\epsilon$. Another way of expressing that
the indicator of a set can be sandwiched between two continuous
functions is that the boundary of the set has zero Lebesgue measure.
\begin{defn}[label={defn:4}]{$\nu$-Riemann integrability}{}\rr
  Let $\nu$ be a finite Borel measure on $\setT$ and let $\phi$ be a
  Borel measurable $\setT\to\setC$ function.

  We call the function $\phi$ \emph{$\nu$-Riemann} integrable if it's
  continuous at $\nu$-almost every point.

  We call the Borel measurable set $B$ \emph{\emph{$\nu$}-Jordan
    measurable} if its indicator function $\setone_B$ is $\nu$-Riemann
  integrable.
\end{defn}

As it is well known, the equivalence of approximability by continuous
functions and the boundary having zero measure carries over to the
setting of any finite Borel measure on the torus.  We can thus extend
the example to the setting when the Lebesgue measure is replaced by an
arbitrary finite Borel measure.

We record our findings in the following result.
\begin{prop}[label={prop:1}]{The Radon-Nikodym derivative can be the
    indicator of a Jordan measurable set}{}\rr
  Let $R$ be a good set, $\alpha$ be an irrational number and let
  $B\subset \setT$ be $\mu_{R,\alpha}$-Jordan measurable with
  $\mu_{R,\alpha}(B)>0$.

  Then the measure $\setone_B\mu_{R,\alpha}$, which is absolutely
  continuous with respect to $\mu_{R,\alpha}$, can be represented at
  $\alpha$ by the good set $S$ defined by
  \begin{equation}
    \label{eq:23}
    S\coloneqq \set*{r}{r\in R, r\alpha\in B}
  \end{equation}
  so we have
  $ \mu_{S,\alpha}=\frac{1}{ \mu_{R,\alpha}(B)}\setone_B
  \mu_{R,\alpha}$.  We also have $\mu_{R,\alpha}(B)=\M_R(S)$.
\end{prop}

Let us go back to trying to represent measures which are absolutely
continuous with respect to the Lebesgue measure $\lambda$. New ideas
are needed to cover the case when we want to represent the measure
$\setone_B\lambda$ when $B$ is a Borel set which is not Jordan
measurable. What is the new difficulty?  We'd like to think that we
could just again take the ``visit set''
$S=\set{n}{n\in\setN,n\alpha\in B}$, but this is not the case anymore.
Indeed, take $B$ to be an open set with $\lambda(B)<1$ and containing
all integer multiples of our irrational $\alpha$. This open set is not
Jordan measurable anymore.  The set $S$ cannot represent the measure
$\setone_B\lambda$ anymore since $S=\setN$. In fact, we show in
\cref{sec:open-set-u} that for any given irrational $\alpha$, one can
construct an open set $B$ so that the visit set of $B$ doesn't even
have mean.  So we definitely need new ideas.

We also need new ideas even for the case when we try to represent a
measure which is absolutely continuous with respect to the Lebesgue
measure with a Radon-Nikodym derivative which is not an indicator
function.  We need these new ideas even if this Radon-Nikodym
derivative is a continuous function.

\section{Proof of \cref{thm:general_representability} for indicators}
\label{sec:proof-}
Strictly speaking, we have already begun the proof of
\cref{thm:general_representability} in the previous section, when we
proved that at an irrational number every measure with Jordan
measurable Radon-Nikodym derivative can be represented.  Our fixed set
up in this section is that we are given a good ``base'' set
$R\subset \setN$ and an irrational number $\alpha$. Since the set $R$
is fixed throughout the section, we suppress the set $R$ from our
notation for the limit measure,
\begin{equation}
  \label{eq:24}
  \mu_\beta\cq \mu_{R,\beta}, \fe \beta
\end{equation}

Since our focus is to widen the class of the Radon-Nikodym derivatives
with respect to the base limit measure $\mu_\alpha$, the following
definition will simplify our language.
\begin{defn}[label={defn:5}]{Representing a function, a Borel
    set}{}\rr
  Let $\rho\in L^1\pa*{\setT,\mu_\alpha}$ be unsigned and
  $\mu_\alpha(\rho) >0$.

  We say that the good set \emph{$S\subset R$ represents $\rho$ at
    $\alpha$} if it represents the measure $\rho\cdot \mu_{\alpha}$,
  that is,
  $\mu_{S,\alpha}=\frac1{\mu_{\alpha}(\rho)}\rho\cdot \mu_{\alpha}$.

  If $\rho$ is the indicator of a Borel measurable set
  $B\subset \setT$, we then say \emph{$S$ represents $B$ at $\alpha$}.
\end{defn}
The sets $S\subset R$ we consider in this section have positive mean
in $R$.  For such a set, the non-normalized averages
$\setA_{n\in[1,N]}\setone_S(r_n)\delta_{r_n\beta}$ are easier to
handle than the normalized ones $\setA_{s\in S(N)}\delta_{s\beta}$.
The convergence or divergence properties of the two averages are
identical since they are connected by
\begin{equation}
  \label{eq:25}
  \lim_N\setA_{n\in[1,N]}\setone_S(r_n)\delta_{r_n\beta}
  =\M_R(S)\lim_N\setA_{s\in S(N)}\delta_{s\beta}
\end{equation}
as can be seen from writing
$\setA_{s\in S(N)}\delta_{s\beta}=\frac{\# R(N)}{\# S(N)}\setA_{r\in
  R(N)}\setone_S(r)\delta_{r\beta}$ and noting that
$\lim_N\frac{\# S(N)}{\# R(N)}=\M_R(S)$ and
$\lim_N \setA_{r\in R(N)}\setone_S(r)\delta_{r\beta}=\lim_N
\setA_{n\in[1,N]}\setone_S(r_n)\delta_{r_n\beta}$.

In \cref{sec:basic-example-repr} we proved that if $B$ is
$\mu_\alpha$-Jordan measurable, then it can be represented by the set
$S_B$ defined by
\begin{equation}
  \label{eq:26}
  S_B=\set*{r_n}{r_n\alpha\in B}
\end{equation}
and we have the relation
\begin{equation}
  \label{eq:27}
  \M_R(S_B)=\mu_\alpha(B)
\end{equation}
We also indicated that this definition of $S_B$ may not give a good
set if $B$ is not Jordan measurable. The idea of extending the
representation to any Borel measurable set is via a limit procedure.
To explain what we mean by ``a limit procedure'', consider the case
when $B$ is an open set, and write it as a disjoint union of open
intervals, $B=\cup_jI_j$. Defining $B_k\cq \cup_{j\in[1,k]}I_j$ for
every $k\in\setN$, each $B_k$ is Jordan measurable and the sequence
$(B_k)$ increases monotonically to $B$. We have
$\lim_k\mu_\alpha(B_k)=\mu_\alpha(B)$. Denoting $S_k\cq S_{B_k}$, the
sequence $(S_k)$ also increases to a set $S\subset R$, but $\M_R(S)$
not only may not be equal $\lim_k\M_R(S_k)$ but $\M_R(S)$ may not even
exist\footnote{See also \cref{sec:open-set-u}.}.  The limit procedure
which is suitable for our purposes is determined by the seminorm
$\none{f}$ which is defined by
\begin{equation}
  \label{eq:28}
  \none{f}\cq \limsup_N\setA_{[1,N]}|f(n)|, \quad f\in \setC^{\setN}
\end{equation}

Our main tools will be two lemmas.  The first one is modeled after a
result of Marcinkiewicz\autocite{marcinkiewicz1939remarque} on the
completeness of Besicovitch spaces.
\begin{lem}[label={lem:1}]{Cauchy sequence is convergent in the
    seminorm $\none{}$}{}\rr
  For each $k\in\setN$, let $f_k\in\setC^{\setN}$.  Suppose that
  $(f_k)$ is a Cauchy sequence in the seminorm $\none{}$, so we have
  \begin{equation}
    \label{eq:29}
    \lim_k\sup_{l\ge k}\none{f_l-f_k}=0
  \end{equation}

  Then there is $f\in \setC^{\setN}$ satisfying
  \begin{equation}
    \label{eq:30}
    \lim_k\none{f_k-f}=0
  \end{equation}
  The $f$ in \cref{eq:30} is pasted together from the $f_k$'s in the
  following way: there are indices $N_1<N_2<\dots$ so that $f=f_k$ on
  the interval $(N_k,N_{k+1}]$,
  \begin{equation}
    \label{eq:31}
    f=\sum_kf_k\cdot\setone_{(N_k,N_{k+1}]}
  \end{equation}
\end{lem}
\begin{rem}[label={rem:1}]{$f$ inherits properties of $(f_k)$}{}
  Since $f$ is pasted together from the $f_k$'s the way we can see it
  in \cref{eq:31}, $f$ inherits some common properties the $f_k$ may
  have.  For example
  \begin{remenum}
  \item \label{item:5} If $f_k\ge 0$ for every $k$ then $f\ge 0$.
  \item \label{item:6} If $|f_k|\le c$ for a constant $c$ for every
    $k$ then $|f|\le c$.
  \item \label{item:7} If each $f_k$ is $0-1$ valued then so is $f$.
  \item \label{item:8} If each $f_k$ is a weight, then the
    construction can be adjusted so that $f$ also becomes a weight.
  \end{remenum}
\end{rem}
Only \cref{item:8} requires some explanation since we
need to have $\sum_{n\in\setN}f(n)=\infty$.  For this, we observe a
flexibility in the choice of the sequence $N_1<N_2<\dots$ in the
upcoming proof of \cref{lem:1}.  Namely the sequence $(N_k)$ is
defined recursively, and once $N_1<N_2<\dots<N_{k-1}$ are chosen, the
index $N_k$, $N_k>N_{k-1}$, is chosen ``large enough'' to satisfy some
criteria.  So it can always be chosen to be ``even larger'' to satisfy
additional criteria.  For our case the single additional criterion is
to ensure $\sum_{n\in (N_{k-1},N_k]}f_{k-1}(n)>1$, which is possible
since $f_{k-1}$ is assumed to be a weight, so
$\sum_{n\in (N_{k-1},\infty)}f_{k-1}(n)=\infty$.
\begin{proof}[Proof of \cref{lem:1}]
  For the recursive definition of the $(N_k)$, define first the
  sequence $(\epsilon_k)$ by
  \begin{equation}
    \label{eq:32}
    \epsilon_k\cq 2\sup_{l\ge k}\none{f_l-f_k}
  \end{equation}
  We can assume, without loss of generality, that $\epsilon_k>0$ for
  every $k$, since $\epsilon_k=0$ for some $k$ would imply
  $\none{f_l-f_k}=0$ for $l\ge k$ hence we could take $f=f_k$.

  In the first step of the recursion, let $N_1=1$.

  In the second step, let $N_2>N_1$ to be large enough to satisfy
  \begin{align}
    \label{eq:33}
    \frac{N_1}{N_2}&<\epsilon_1\\
    \setA_{[1,N]}\abs*{f_1-f_2}
                   &<\epsilon_1 \fe N\ge N_2\\
    \intertext{and}
    \setA_{[1,N]}\abs*{f_1-f_3}
                   &<\epsilon_1 \fe N\ge N_2
  \end{align}
  Complete the second step of the recursion by defining $f$ to be
  equal $f_1$ on the interval $(N_1,N_2]$.  Let $k> 2$ and assume that
  we have defined $N_1<N_2<\dots<N_{k-1}$ and $f$ to be equal $f_j$ on
  the interval $(N_j,N_{j+1}]$ for $j\in[1,k-2]$. For step $k$ of the
  recursion let $N_{k}>N_{k-1}$ be large enough to satisfy
  \begin{align}
    \label{eq:34}
    \frac1{N_k}\sum_{
    [1,N_{k-1}]}\abs*{f_j-f}
    &<\epsilon_j, \fe j\in [1,k-2]\\
    \setA_{[1,N]}\abs*{f_j-f_{k-1}}
    &<\epsilon_j \fe N\ge N_k,  j\in [1,k-2]\label{eq:35}\\
    \intertext{and}
    \setA_{[1,N]}\abs*{f_j-f_k}
    &<\epsilon_j \fe N\ge N_k,  j\in [1,k-2]\label{eq:36}
  \end{align}
  Complete the $k$th step of the recursion by defining $f$ to be equal
  $f_{k-1}$ on the interval $(N_{k-1},N_k]$.

  Let us fix $j$ and let $N$ be large enough so that for some
  $k\ge j+2$ we have
  \begin{equation}
    \label{eq:37}
    N_k\le N<N_{k+1}
  \end{equation}
  We want to show that
  \begin{equation}
    \label{eq:38}
    \setA_{[1,N]}\abs*{f_j-f}<3\epsilon_j
  \end{equation}
  Let us estimate $\setA_{[1,N]}\abs*{f_j-f}$ as,
  \begin{align}
    \label{eq:39}
    \setA_{[1,N]}\abs*{f_j-f}
    &=\frac1N\sum_{[1,N_{k-1}]}\abs*{f_j-f}\\
    &+\frac1N\sum_{(N_{k-1},N_k]}\abs*{f_j-f}\label{eq:40}\\
    &+\frac1N\sum_{(N_{k},N]}\abs*{f_j-f}\label{eq:41}                                                     
  \end{align}
  We can estimate the term in \cref{eq:39}, using \cref{eq:34} and
  that $N\ge N_k$, as
  \begin{equation}
    \label{eq:42}
    \frac1N\sum_{[1,N_{k-1}]}\abs*{f_j-f}<\epsilon_j
  \end{equation}
  For the term in \cref{eq:40} we have
  \begin{equation}
    \label{eq:43}
    \frac1N\sum_{(N_{k-1},N_k]}\abs*{f_j-f}<\epsilon_j
  \end{equation}
  This follows from \cref{eq:35} since $f=f_{k-1}$ on the interval
  $(N_{k-1},N_k]$.

  For the term in \cref{eq:41} we have
  \begin{equation}
    \label{eq:44}
    \frac1N\sum_{(N_k,N]}\abs*{f_j-f}<\epsilon_j
  \end{equation}
  This follows from \cref{eq:36} since $f=f_{k}$ on the interval
  $(N_k,N]$.
  
  Putting the estimates in \cref{eq:42,eq:43,eq:44} together we obtain
  \cref{eq:38}.
\end{proof}
The second lemma shows that the family $\mathcal{M}$ of sequences $f$
for which $\M(f)=\lim_N\setA_{[1,N]}f$ exists is closed with respect
to the upper mean $\nmean{ }$ defined by
\begin{equation}
  \label{eq:45}
  \nmean{f}\cq \limsup_N\abs*{\setA_{n\in[1,N]}f(n)},
  \quad f\in \setC^{\setN}
\end{equation}

\begin{lem}[label={lem:2}]{$\mathcal{M}$ is closed with respect to
    $\nmean{}$}{}
  Let $(f_j)$ be a sequence from $\mathcal{M}$.  Suppose that $(f_j)$
  converges to $f\in \setC^{\setN}$ in the seminorm $\nmean{}$, so
  \begin{equation}
    \label{eq:46}
    \lim_j\nmean{f_j-f}=0
  \end{equation}
  Then $f\in \mathcal{M}$ and
  \begin{equation}
    \label{eq:47}
    \M(f)=\lim_j \M(f_j)
  \end{equation}
\end{lem}
\begin{proof}
  First note that, as a consequence of \cref{eq:46}, the sequence
  $(f_j)$ is a Cauchy sequence, meaning that for a given $\epsilon>0$
  there is $J$ so that
  \begin{equation}
    \label{eq:48}
    \nmean{f_j-f_J}<\epsilon\fe j\ge J
  \end{equation}
  Since $\abs*{\M(f_j)-\M(f_J)}=\abs*{\M(f_j-f_J)}=\nmean{f_j-f_J}$ we
  see,
  \begin{equation}
    \label{eq:49}
    \abs*{\M(f_j)-\M(f_J)}<\epsilon\fe j\ge J
  \end{equation}
  so the sequence $\M(f_j)$ of means is a Cauchy sequence of numbers.
  Denote $L\coloneqq \lim_j\M(f_j)$.  We want to show that $\M(f)=L$.
  For a given $\epsilon>0$, choose a $j$ so that
  $\abs*{\M(f_j)-L}<\epsilon$ and $\nmean{f-f_j}<\epsilon$.  We then
  have, for an arbitrary $N$,
  \begin{equation}
    \label{eq:50}
    \abs*{\setA_{[1,N]}f-L}\le \abs*{\setA_{[1,N]}(f-f_j)}
    +\abs*{\setA_{[1,N]}f_j-L}
  \end{equation}
  Taking $\limsup_N$ of both sides, we get
  \begin{equation}
    \label{eq:51}
    \limsup_N\abs*{\setA_{[1,N]}f-L}\le \nmean{f-f_j}+\abs*{\M(f_j)-L}
  \end{equation}
  Since $\nmean{f-f_j}<\epsilon$ and $\abs*{\M(f_j)-L}<\epsilon$, we
  get $\limsup_N\abs*{\setA_{[1,N]}f-L}<2\epsilon$.  Since
  $\epsilon>0$ was arbitrary, we have
  $\lim_N\abs*{\setA_{[1,N]}f-L}=0$ which means
  $\M(f)=L=\lim_j\M(f_j)$.

\end{proof}
How do we now show that every open set can be represented?  Let
$B\subset \setT$ be open with positive $\mu_\alpha$ measure, let
$B=\cup_jI_j$ be its decomposition into pairwise disjoint open
intervals $I_j$ and set $B_k\cq \cup_{j\in[1,k]}I_j$.  Since
$\mu_\alpha(B)>0$, we have $\mu_\alpha(B_k)>0$ for large enough $k$.
For simplicity, we assume that $\mu_\alpha(B_k)>0$ for every $k$. The
sets $B_k$ increase to $B$ monotonically, hence, in particular, we
have $\lim_k\mu_\alpha\pa*{B_k\triangle B}=0$. According to
\cref{prop:1}, the set $B_k$ can be represented by the set
$S_k\subset R$ defined by
\begin{equation}
  \label{eq:52}
  S_k\cq \set*{r_n}{r_n\alpha \in B_k}
\end{equation}
and we have $\M_R(S_k)=\mu_\alpha(B_k)$. Since for every $k,l$ the set
$B_k\triangle B_l$ is Jordan measurable, we also have
\begin{equation}
  \label{eq:53}
  \M_R\pa*{S_k\triangle
    S_l}=\mu_\alpha\pa*{B_k\triangle B_l}
\end{equation}
For each $k$ let us define the sequence $f_k$ by
\begin{equation}
  \label{eq:54}
  f_k(n)\cq\setone_{S_k}(r_n),\qquad n\in\setN
\end{equation}
We have
\begin{equation}
  \label{eq:55}
  \M(f_k)=\M_R(S_k)\fe k\in\setN
\end{equation}
and we can rewrite \cref{eq:53} as
\begin{equation}
  \label{eq:56}
  \M\abs*{f_k-f_l}=\mu_\alpha\pa*{B_k\triangle B_l}
\end{equation}
Since $(B_k)$ is a Cauchy sequence, so
$\lim_k\sup_{l\ge k}\mu_\alpha\pa*{B_k\triangle B_l}=0$, \cref{eq:56}
implies that $(f_k)$ is also a Cauchy sequence in $\none{}$, so we
have $\lim_k\sup_{l\ge k}\M\abs*{f_k-f_l}=0$. Since
$\M\abs*{f_k-f_l}=\none{f_k-f_l}$, according to \cref{lem:1}, there is
$f$ to which the $(f_k)$ converges, that is, so
$\lim_k\none{f_k-f}=0$, and by \cref{lem:2}, $\M(f)=\lim_k \M(f_k)$.
Since $\M(f_k)=\mu_{\alpha}(B_k)$ and
$\lim_k\mu_{\alpha}(B_k)=\mu_{\alpha}(B)$, we have
$\M(f)=\mu_{\alpha}(B)>0$. According to \cref{item:7}
$f$ is $0-1$ valued hence we can define a set $S\subset R$ by its
indicator as
\begin{equation}
  \label{eq:57}
  \setone_S(r_n)\cq f(n),\qquad n\in\setN
\end{equation}
We have
\begin{equation}
  \label{eq:58}
  \M_R(S)=\M(f)
\end{equation}
We want to show that $S$ is good and it represents $B$ at $\alpha$.
To this end, let $\beta\in\setT$ be arbitrary and define the sequences
$f^\beta_k$ and $f^\beta$ by
\begin{align}
  \label{eq:59}
  f^\beta_k(n)&\cq f_k(n)\e(r_n\beta) \text{ for }n\in\setN\\
  f^\beta(n)&\cq f(n)\e(r_n\beta)\text{ for }n\in\setN
\end{align}
Since $\M(f)=\lim_k \M(f_k)>0$ and $\M(f_k)=\M_R(S_k)$, we have
$\M(S)>0$.  It follows that, by \cref{eq:25}, to show that $S$ is
good, it's enough to show that $\M(f^\beta)$ exists for every $\beta$
and to show that $S$ represents $B$ at $\alpha$ it's enough to show
that $\M(f^{p\alpha})=\mu_\alpha\pa*{\e^p\setone_B}$ for every
$p\in\setZ$.

Let us first show that $\M(f^\beta)$ exists for every $\beta$.  Since
each set $S_k$ is good with $\M(S_k)>0$, we have, as a consequence of
\cref{eq:25}, that $f^\beta_k\in\mathcal{M}$ for every $k,\beta$.  The
fact that for every $\beta$, the sequence $(f^\beta_k)$ converges to
$f^\beta$ in the norm $\nmean{}$ follows from the uniform estimate
\begin{equation}
  \label{eq:60}
  \nmean{f^\beta_k-f^\beta}\le \none{f_k-f} \fe \beta
\end{equation}

By \cref{lem:2}, $f^\beta\in\mathcal{M}$ and
\begin{equation}
  \label{eq:61}
  \M(f^\beta)=\lim_k\M(f^\beta_k)
\end{equation}
Let us now show that $S$ represents $B$ at $\alpha$, that is,
$\M(f^{p\alpha})=\mu_\alpha\pa*{\e^p\setone_B}$ for every
$p\in\setZ$. Since the sequence $\pa*{B_k}$ converges to $B$ in
$L^1(\mu_\alpha)$-norm we have
\begin{equation}
  \label{eq:62}
  \lim_k\mu_\alpha\pa*{\e^p\setone_{B_k}}=\mu_\alpha\pa*{\e^p\setone_{B}}
  \fe p\in\setZ
\end{equation}
Since $\M(f^{p\alpha}_k)=\mu_\alpha\pa*{\e^p\setone_{B_k}}$ and, by
\cref{eq:61}, $\lim_k\M(f^{p\alpha}_k)=\M(f^{p\alpha})$, \cref{eq:62}
implies that $\M(f^{p\alpha})=\mu_\alpha\pa*{\e^p\setone_{B}}$.

We record the general idea we used as \cref{item:10} below.
\begin{prop}[label={prop:2}]{Limit of good sets with positive mean is
    good}{}\rr
  Let $(S_k)$ be a sequence of good subsets of $R$ with mean which
  converge to $S\subset R$ in $\noner{}$-seminorm, that is,
  $\lim_k\noner{S_k\triangle S}=0$.  Assume that $\limsup_k\M_R(S_k)>0$.

  Then we have the following.
  \begin{propenum}
  \item \label{item:9} $\lim_k\M_R(S_k)$ exists and
    $\M_R(S)=\lim_k\M_R(S_k)>0$.
  \item \label{item:10} $S$ is a good set.
  \item \label{item:11} The sequence $\pa*{\mu_{S_k,\beta}}_k$ of limit
    measures converge to $\mu_{S,\beta}$ in variation distance and
    uniformly in $\beta$,
    \begin{equation}
      \label{eq:63}
      \lim_k\sup_\beta\nvar{\mu_{S_k,\beta}-\mu_{S,\beta}} =0
    \end{equation}
  \item \label{item:12} Let $\nu$ be a Borel measure on $\setT$.

    If for some $\alpha$, $\mu_{S_k,\alpha}$ is absolutely continuous
    with respect to $\nu$ with Radon-Nikodym derivative $\rho_{k}$ for
    every $k$, then $\mu_{S,\alpha}$ is also absolutely continuous
    with respect to $\nu$ with Radon-Nikodym derivative $\rho$ which
    satisfies
    \begin{equation}
      \label{eq:64}
      \lim_k\norm*{\rho_{k}-\rho}_{L^1\pa*{\nu}}=0
    \end{equation}
  \end{propenum}
\end{prop}
 
 \begin{proof}
   The proof of \cref{item:9} follows from the triangle inequality for
   the $\none{}$-seminorm, since we then have
   \begin{align*}
     \abs*{\M_R(S_k)-\M_R(S)}&=\abs*{\noner{S_k}-\noner{S}}\\
                         &\le \noner{S_k\triangle S}
   \end{align*}
   and just use the assumption that $\lim_k \noner{S_k\triangle S}=0$.

   The argument we gave just before the enunciation of our proposition
   proves that $S$ is a good set.
   
   For the proof of \cref{item:11} note that in the argument preceding
   our proposition we proved that the sequence
   $\pa*{\mu_{S_k,\beta}}_k$ of measures converges weakly to
   $\mu_{S,\beta}$ for every $\beta$ but an estimate similar to
   \cref{eq:60} enables us to draw the stronger conclusion of
   \cref{eq:63}.

   The following lemma gives us the estimates we need.
   \begin{lem}[label={lem:3}]{$\noner{}$ dominates $\nvar{}$ and
       $\norm{}_{L^1}$}{}
     Let $v_1,v_2$ be good $R$-weights.  Assume that
     \begin{equation}
       \label{eq:65}
       \max \cbrace*{\noner{v_1},\noner{v_2}}>0
     \end{equation}
     Then we have the following.
     \begin{lemenum}
     \item \label{item:13}
       \begin{equation}
         \label{eq:66}
         \sup_\beta\nvar{\mu_{v_1,\beta}-\mu_{v_2,\beta}}
         \le\frac{2}{\max \cbrace*{\noner{v_1},\noner{v_2}}}\noner{v_1-v_2}        
       \end{equation}
     \item \label{item:14}If, for some $\alpha$, the limit measures
       $\mu_{v_1,\alpha}$ and $\mu_{v_2,\alpha}$ are absolutely
       continuous with respect to a Borel measure $\nu$ on $\setT$
       with Radon-Nikodym derivatives $\rho_{1}$ and $\rho_2$,
       respectively, then
       \begin{equation}
         \label{eq:67}
         \norm*{\rho_{1}-\rho_{2}}_{L^1\pa*{\nu}}
         \le \frac{4}{\max \cbrace*{\noner{v_1},\noner{v_2}}}\noner{v_1-v_2}
       \end{equation} 
     \end{lemenum}
   \end{lem}
   \begin{proof}
     To prove \cref{item:13}, that is, the inequality in \cref{eq:66},
     fix $\beta$ and $\phi\in \mathcal{C}_+$, so $\phi$ is a
     continuous $\setT\to\setC$ function with $0\le \phi\le 1$. We can
     assume without loss of generality that
     $\max \cbrace*{\noner{v_1},\noner{v_2}}=\noner{v_1}$. Let $(N_l)_l$
     be a strictly increasing sequence of indices so that
     \begin{equation}
       \label{eq:68}
       \lim_l \setA_{[1,N_l]}v_1=\noner{v_1}
     \end{equation}
     Let us estimate as
     \begin{align*}
       &\abs*{\setA^{v_1}_{n\in[1,N_l]}\phi(r_n\beta)-\setA^{v_2}_{n\in[1,N_l]}\phi(r_n\beta)}\\
       &\qquad=\abs*{\frac{1}{\setA_{[1,N_l]}v_1}\setA_{n\in[1,N_l]}v_1(r_n)\phi(r_n\beta)
         -\frac{1}{\setA_{[1,N_l]}v_2}\setA_{n\in[1,N_l]}v_2(r_n)\phi(r_n\beta)}\\
       \intertext{ adding
       $0=-\frac{1}{\setA_{[1,N_l]}v_1}\setA_{n\in[1,N_l]}v_2(r_n)\phi(r_n\beta)
       +\frac{1}{\setA_{[1,N_l]}v_1}\setA_{n\in[1,N_l]}v_2(r_n)\phi(r_n\beta)$
       inside the absolute value and using the triangle inequality,}
       &\qquad\le \frac{1}{\setA_{[1,N_l]}v_1}
         \abs*{\setA_{n\in[1,N_l]}v_1(r_n)\phi(r_n\beta)
         -\setA_{n\in[1,N_l]}v_2(r_n)\phi(r_n\beta)}\\
       &\qquad\quad+\abs*{\frac{1}{\setA_{[1,N_l]}v_1}-\frac{1}{\setA_{[1,N_l]}v_2}}
         \abs*{\setA_{n\in[1,N_l]}v_2(r_n)\phi(r_n\beta)}\\
       &\qquad\le \frac{1}{\setA_{[1,N_l]}v_1}\setA_{[1,N_l]}\abs*{v_1-v_2}
         +\frac{\setA_{[1,N_l]}\abs*{v_1-v_2}}
         {\setA_{[1,N_l]}v_1\setA_{[1,N_l]}v_2}\setA_{[1,N_l]}v_2\\
       &\qquad=\frac{2}{\setA_{[1,N_l]}v_1}\setA_{[1,N_l]}\abs*{v_1-v_2}
     \end{align*}
     so we have
     \begin{equation}
       \label{eq:69}
       \abs*{\setA^{v_1}_{n\in[1,N_l]}\phi(r_n\beta)-\setA^{v_2}_{n\in[1,N_l]}\phi(r_n\beta)}
       \le \frac{2}{\setA_{[1,N_l]}v_1}\setA_{[1,N_l]}\abs*{v_1-v_2}
     \end{equation}
     Since
     $\lim_l\setA^{v_i}_{n\in[1,N_l]}\phi(r_n\beta)=\mu_{v_i,\beta}(\phi)$,
     $\lim_l\setA_{[1,N_l]}v_1=\noner{v_1}$ and
     $\limsup_l\frac{2}{\setA_{[1,N_l]}v_1}\setA_{[1,N_l]}\abs*{v_1-v_2}\le
     \frac2{\noner{v_1}}\noner{v_1-v_2}$, we get
     \begin{equation}
       \label{eq:70}
       \abs*{\mu_{v_1,\beta}(\phi)-\mu_{v_2,\beta}(\phi)}\le \frac2{\noner{v_1}}\noner{v_1-v_2}
     \end{equation}
     which is independent of $\beta$ and $\phi\in\mathcal{C}_+$,
     proving \cref{eq:66}.

     To prove \cref{item:14}, observe first that,
     since $\mu_{v_i,\alpha}=\rho_i\nu$ and $\rho_i$ are probability
     densities with respect to $\nu$, we have
     $\nvar{\rho_1\nu-\rho_2\nu}=\frac12\norm*{\rho_{1}-\rho_{2}}_{L^1\pa*{\nu}}$.
     It follows that
     \begin{equation}
       \label{eq:71}
       \nvar{\mu_{v_1,\alpha}-\mu_{v_2,\alpha}}
       =\frac12\norm*{\rho_{1}-\rho_{2}}_{L^1\pa*{\nu}}
     \end{equation}
     and now just  use \cref{eq:66}.
   \end{proof}

   Now, let us come back to the proof of \cref{prop:2}. Using
   \cref{eq:66} with $v_1=\setone_{S_k}$ and $v_2=\setone_S$, we get
   \begin{equation}
     \label{eq:72}
     \sup_\beta\nvar{\mu_{S_k,\beta}-\mu_{S,\beta}}
     \le\frac{2}{\max \cbrace*{\noner{S_k},\noner{S}}}\noner{S_k\triangle S} 
   \end{equation}
   Using the assumption that $\lim_k\noner{S_k\triangle S}=0$ and that,
   by \cref{item:9}, we have
   $\lim_k\noner{S_k}=\lim_k\M_R(S_k)=\M_R(S)=\noner{S}>0$, we get
   \cref{eq:63}.

   For the proof of \cref{item:12}, by \cref{item:9}, we can assume,
   without loss of generality that $\M_R(S_k)>0$ for every $k$. Using
   \cref{eq:67} with $v_1=\setone_{S_k}$ and $v_2=\setone_{S_l}$ we
   get
   \begin{equation}
     \label{eq:73}
     \norm*{\rho_{l}-\rho_{k}}_{L^1\pa*{\nu}}
     \le \frac{4}{\max
       \cbrace*{\noner{S_k},\noner{S_l}}}\noner{S_k\triangle S_l}
   \end{equation}
   This implies, since the sequence $(S_k)$ is convergent in
   $\noner{}$-seminorm and hence is Cauchy, that the sequence
   $(\rho_k)$ is Cauchy in $L^1(\nu)$-norm.  Since $L^1(\nu)$ is
   complete and $\nu(\rho_k)=1$ for every $k$, there is a
   $\rho\in L^1(\nu)$ with $\nu(\rho)=1$ so that
   \begin{equation}
     \label{eq:74}
     \lim_k\norm*{\rho_{k}-\rho}_{L^1\pa*{\nu}}=0
   \end{equation}
   
   Since
   $\norm*{\rho_{k}-\rho}_{L^1\pa*{\nu}}=2\nvar{\rho_{k}\nu-\rho\nu}$
   and $\rho_{k}\nu=\mu_{S_k,\alpha}$, we get
   \begin{equation}
     \label{eq:75}
     \lim_k\nvar{\mu_{S_k,\alpha}-\rho\nu}=0
   \end{equation}
   But by \cref{item:11} we also have
   $\lim_k\nvar{\mu_{S_k,\alpha}-\mu_{S,\alpha}}=0$ hence we must have
   $\mu_{S,\alpha}=\rho\nu$.
   
 \end{proof}

 We can use \cref{prop:2} in an argument similar to the one we used to
 show that any open set can be represented at $\alpha$ to prove that
 if a $G_\delta$ set $B$ has positive $\mu_\alpha$-measure then it can
 be represented at $\alpha$. Only the initial setup of the proof is
 different.  This time let $(B_k)$ be a \emph{decreasing} sequence of
 open sets which converges to $B$.  Let $S_k\subset R$ represent $B_k$
 at $\alpha$. We again have the isometry \cref{eq:53} from which
 everything follows: the existence of a good set $S$ which represents
 $B$ at $\alpha$ and $\M(S)=\mu_\alpha(B)$.

 Since every Borel measurable set differs from a $G_\delta$ set on a
 set of $\mu_\alpha$-measure zero, we in fact showed that every Borel
 set of positive $\mu_\alpha$-measure can be represented.  So we
 proved the following more precise version of
 \cref{thm:general_representability} for the case when the
 Radon-Nikodym derivative of a measure with respect to $\mu_\alpha$ is
 an indicator.
 \begin{prop}[label={prop:3}]{\Cref{thm:general_representability} for
     indicators}{}\rr
   Let $R\subset \setN$ be a good set, $\alpha$ be an irrational
   number, and let $B$ be a Borel set with $\mu_\alpha(B)>0$.

   Then $B$ can be represented at $\alpha$ by a set $S\subset R$ which
   satisfies
   \begin{equation}
     \label{eq:76}
     \M_R(S)=\mu_\alpha(B)>0
   \end{equation}
 \end{prop}

 \section{Measures that cannot be represented at every irrational
   $\alpha$}
 \label{sec:measures-that-cannot}

 For this section, we suspend the proof of
 \cref{thm:general_representability} just to see how \cref{prop:3} can
 be used to prove \cref{thm:cont_only}.  We will also prove
 \cref{cantor_nonrep}.

 \subsection{Proof of \cref{thm:cont_only}}
 \label{sec:proof-crefthm:c}
 In this section we want to prove that if the Borel probability
 measure $\nu$ has a point-mass at a point $\gamma\in\setT$ and
 $\alpha$ is irrational then $\nu$ cannot be represented at $\alpha$.

 The proof is by contradiction: let us assume that for some
 $\gamma\in\setT$, $\nu(\{\gamma\})>0$ and that $\nu$ can be
 represented by the set $R$ at $\alpha$, so $\mu_{R,\alpha}=\nu$. Then
 the Dirac mass $\delta_{\gamma}$ is absolutely continuous with
 respect to $\mu_{R,\alpha}$ with Radon-Nikodym derivative equal
 $\frac1{\nu(\gamma)}\setone_{\{\gamma\}}$.  By \cref{prop:3} there is
 a good set $S\subset R$ which represents $\delta_{\gamma}$ at
 $\alpha$, so $\mu_{S,\alpha}=\delta_{\gamma}$. Let us define the
 function $\phi:\setT\to\setC$ as
 \begin{equation}
   \label{eq:77}
   \phi(\beta)\cq \mu_{S,\beta}(\e)
 \end{equation}
 Then, by the definition of $\mu_{S,\beta}(\e)$, $\phi$ is the limit
 of the sequence $(\phi_N)$ of continuous functions defined by
 $\phi_N(\beta)\cq \setA_{n\in [1,N]}\e(s_n\beta)$ where $(s_n)$ is
 the elements of $S$ arranged in increasing order.  Since for every
 $p\in\setZ$ we have $\mu_{S,p\alpha}(\e)=\mu_{S,\alpha}(\e^p)$ and
 $\mu_{S,\alpha}(\e^p)=\e^p(\gamma)$, we have
 \begin{equation}
   \label{eq:78}
   \abs*{\phi}=1 \text{ on the dense set }\set*{p\alpha}{p\in\setZ}
 \end{equation}
 By Weyl's theorem\autocites[Satz 21]{MR1511862}[Theorem
 4.1]{MR0419394}, $\phi=0$ on a set of full Lebesgue measure, so, as a
 consequence,
 \begin{equation}
   \label{eq:79}
   \phi=0 \text{ on a dense set}.
 \end{equation}
 By Baire's theorem\autocite[Page 83]{zbMATH01030714}, \cref{eq:78,eq:79}
 together are impossible to hold simultaneously for the limit of
 continuous functions.

 \subsection{Proof of \cref{cantor_nonrep}}
 \label{sec:proof--1}
 So in this section we want to prove that if $\nu$ is a Borel
 probability measure on $\setT$ with
 $\limsup_{p\to\infty}\abs*{\nu(\e^p)}>0$ then there is an irrational
 $\alpha$ where $\nu$ cannot be represented. In fact the set of such
 $\alpha$'s is of full Lebesgue measure.

 From the assumption that $\limsup_{p\to\infty}\abs*{\nu(\e^p)}>0$ it
 follows that there is an $\epsilon>0$ and a infinite sequence
 $p_1<p_2<\dots$ of indices so that
 \begin{equation}
   \label{eq:80}
   |\nu\e^{p_k}|>\epsilon\text{ for }k\in\setN
 \end{equation}
 By Weyl's result\autocites[Satz 21]{MR1511862}[Theorem
 4.1]{MR0419394}, the set $A\subset \setT$ defined by
 \begin{equation}
   \label{eq:81}
   A\coloneqq
   \set*{\alpha}{\overbar{\set{p_k\alpha}{k\in\setN}}\text{ has
       nonempty interior}\pmod 1}
 \end{equation}
 has full $\lambda$ measure.  We want to show that $A$ is a subset of
 those $\alpha$'s at which the measure $\nu$ cannot be represented.

 Let $\alpha\in A$, and suppose the measure $\nu$ can be represented
 at $\alpha$, say, by the set $S=(s_n)$, that is,
 $\mu_{S,\alpha}=\nu$.  Let us define the function
 $\phi:\setT\to\setC$ as
 \begin{equation}
   \label{eq:82}
   \phi(\beta)\cq \mu_{S,\beta}(\e)
 \end{equation}
 Then, by the definition of $\mu_{S,\beta}(\e)$, $\phi$ is the limit
 of the sequence $(\phi_N)$ of continuous functions defined by
 $\phi_N(\beta)\cq \setA_{n\in[1,N]}\e(s_n\beta)$.  Since for every
 $p\in\setZ$ we have $\mu_{S,p\alpha}(\e)=\mu_{S,\alpha}(\e^p)$ and
 $\mu_{S,\alpha}(\e^p)=\nu\pa*{\e^p}$, by \cref{eq:80} we have
 \begin{equation}
   \label{eq:83}
   \abs*{\mu_{S,p_k\alpha}(\e)}>\epsilon \fe k\in\setN
 \end{equation}
 By the definition of $\phi$, we can write the above as
 \begin{equation}
   \label{eq:84}
   \abs*{\phi}>\epsilon \text{ on the set }\set*{p_k\alpha}{k\in\setN}
 \end{equation}
 Since $\alpha\in A$, the set $\set*{p_k\alpha}{k\in\setN}$ is dense
 in a nondegenerate interval $I\subset \setT$.
 
 By Weyl's theorem\autocites[Satz 21]{MR1511862}[Theorem
 4.1]{MR0419394}, $\phi=0$ on a set $U$ of full Lebesgue measure
 \begin{equation}
   \label{eq:85}
   \phi=0 \text{ on } U
 \end{equation}
 Since both $\set*{p_k\alpha}{k\in\setN}$ and $U$ are dense in the
 interval $I$, by Baire's theorem\autocite[Page 83]{zbMATH01030714},
 \cref{eq:84,eq:85} cannot be true together for the limit $\phi$ of
 continuous functions.

 \section{Representing by weights}
 \label{sec:representing-weights}

 In this section, we fix the good set\footnote{Note that we make no
   further assumption on $R$, such as sublacunarity} $R$ and the
 irrational number $\alpha$, and we continue in the tradition of
 \cref{sec:proof-} suppressing the set $R$ in our notation for the
 limit measure, so $\mu_\alpha=\mu_{R,\alpha}$.

 In trying to extend the class of representable functions $\rho$ from
 indicators, we first consider an easier problem.  Instead of
 representing by sets, we represent by $R$-weights.
 \begin{defn}[label={defn:6}]{Function represented by a weight}{}\rr
   Let $\rho$ be an unsigned $L^1\pa*{\setT,\mu_\alpha}$ function with
   $\mu_\alpha(\rho)>0$.

   We say the $R$-weight $w$ \emph{represents $\rho$ at $\alpha$} if
   $w$ is good and it represents the measure $\rho\cdot \mu_{\alpha}$,
   that is,
   $\mu_{w,\alpha}=\frac1{\mu_{\alpha}(\rho)}\rho\cdot \mu_{\alpha}$.

\end{defn}
The $R$-weights $w$ we consider in this section have positive mean in
$R$, so $\M_R(w)>0$.  For such a weight, the non-normalized averages
$\setA_{n\in[1,N]}w(r_n)\delta_{r_n\beta}$ are easier to handle than
the normalized ones $\setA_{n\in[1,N]}^w\delta_{r_n\beta}$.  The
convergence or divergence properties of the two averages are identical
since they differ only by the nonzero factor $\M_R(w)$,
\begin{equation}
  \label{eq:86}
  \lim_N\setA_{n\in[1,N]}w(r_n)\delta_{r_n\beta}=\M_R(w)\lim_N\setA_{n\in[1,N]}^w\delta_{r_n\beta}
\end{equation}
as can be seen from writing
$\setA_{n\in[1,N]}^w\delta_{r_n\beta}=\frac{N}{\sum_{n\in[1,N]}w(r_n)}\setA_{n\in
  [1,N]}w(r_n)\delta_{r\beta}$.

In \cref{sec:basic-example-repr} we have already seen that if $\rho$
is an unsigned continuous function with $\mu_\alpha(\rho)>0$ then the
$R$-weight $w$ defined by
\begin{equation}
  \label{eq:87}
  w(r_n)\cq \rho(r_n\alpha)
\end{equation}
is good, unsigned and it represents $\rho$ at $\alpha$. Since every
unsigned $\mu_\alpha$-integrable function can be approximated
arbitrary closely by unsigned continuous functions in
$L^1\pa*{\setT,\mu_\alpha}$-norm, the proof of \cref{item:3} requires
only an approximation argument similar to what we had in
\cref{sec:proof-}.  We restate \cref{item:3} in the following form for
the readers convenience.

\begin{prop}[label={prop:4}]{Any integrable function is
    representable with weights}{}\rr 
  Let $\rho$ be an unsigned function from $L^1\pa*{\setT,\mu_\alpha}$
  with $\mu_\alpha(\rho)>0$.

  Then there is an $R$-weight $w$ which represents $\rho$ at $\alpha$.
  In particular, we have
  \begin{equation}
    \label{eq:88}
    \M_R(w)=\mu_\alpha(\rho)
  \end{equation}
  Furthermore, if $\rho$ is a bounded function then the representing
  $R$-weight $w$ can be chosen to be bounded.
\end{prop}
The proof of \cref{prop:2} can be easily adjusted to obtain the
following analog for weights.
\begin{prop}[label={prop:5}]{Limit of good weights with positive mean
    is good}{}\rr
  Let $(w_k)$ be a sequence of good $R$-weights with mean which
  converge to the $R$-weight $w$ in $\noner{}$-seminorm, so
  $\lim_N\noner{w_k- w}=0$.  Assume that $\limsup_k\M_R(w_k)>0$.

  Then we have the following.
  \begin{propenum}
  \item \label{item:15} $\lim_k\M_R(w_k)$ exists and
    $\lim_k\M_R(w_k)=\M_R(w)>0$.
  \item \label{item:16} $w$ is a good $R$-weight.
  \item \label{item:17} The sequence $\pa*{\mu_{w_k,\beta}}_k$ of
    limit measures converge to $\mu_{w,\beta}$ in variation distance
    and uniformly in $\beta$,
    \begin{equation}
      \label{eq:89}
      \lim_k\sup_\beta\nvar{\mu_{w_k,\beta}-\mu_{w,\beta}} =0
    \end{equation}
  \item \label{item:18} Let $\nu$ be a Borel measure on $\setT$.

    If for some $\alpha$, $\mu_{w_k,\alpha}$ is absolutely continuous
    with respect to $\nu$ with Radon-Nikodym derivative $\rho_{k}$ for
    every $k$ then $\mu_{w,\alpha}$ is also absolutely continuous with
    respect to $\nu$ with Radon-Nikodym derivative $\rho$ which
    satisfies
    \begin{equation}
      \label{eq:90}
      \lim_k\norm*{\rho_{k}-\rho}_{L^1\pa*{\nu}}=0
    \end{equation}
  \end{propenum}
\end{prop}
With this proposition, we can complete the proof of \cref{prop:4}
exactly as we proved \cref{prop:3}, using a sequence $(\rho_k)$ of
unsigned continuous functions that converge to $\rho$ in
$L^1\pa*{\mu_\alpha}$-norm.  We need to remark only that if $\rho$ is
a bounded function, then the sequence $(\rho_k)$ of continuous
functions can be chosen to be uniformly bounded.

\section{Proof of \cref{thm:general_representability} for bounded
  $\rho$}
\label{sec:proof-bounded-rho}

In this section, we still are working with a fixed good set $R$ of
positive integers, an irrational number $\alpha$, but now we also fix
a bounded Borel measurable, unsigned function $\rho$ with
$\mu_\alpha(\rho)>0$.  We proved in \cref{sec:representing-weights}
that $\rho$ can be represented at $\alpha$ by a good, bounded
$R$-weight $w$. In this section we will show that there is a good set
$S\subset R$ which also represents $\rho$ at $\alpha$, hence proving
\cref{thm:general_representability} for bounded $\rho$. It follows
from the definition of representation that if the good $R$-weight $w$
represents $\rho$ then so does the $R$-weight $cw$ for every positive
constant $c$.  In particular, we can assume that the $R$-weight $w$
representing $\rho$ is bounded by $1$. We will show that then there is
a set $S\subset R$ so that
\begin{equation}
  \label{eq:91}
  \lim_N\sup_\beta\abs*{\setA_{n\in[1,N]}\setone_S(r_n)\e(r_n\beta)
    -\setA_{n\in[1,N]}w(r_n)\e(r_n\beta)}=0
\end{equation}
The ``construction'' of $S$ satisfying \cref{eq:91} is done randomly.
Our random method requires that we limit the growth of the set $R$; we
need to assume that $R$ is \emph{sublacunary}\autocite[Theorem
B]{MR1721622}.  We need the concept of a sublacunary weight.
\begin{defn}[label={defn:7}]{Sublacunary weight}{}\rr
  The $R$-weight $w$ is called \emph{sublacunary} if it satisfies
  \begin{equation}
    \label{eq:92}
    \lim_N\frac{w\pa*{R(N)}}{\log N}=\infty
  \end{equation}
\end{defn}
We often consider the sequence $(r_n)$ instead of the set $R$ in which
case we can use the following more convenient version of \cref{eq:92}.
\begin{equation}
  \label{eq:93}
  \lim_N\frac{\sum_{n\in[1,N]}w(r_n)}{\log r_{N+1}}=\infty
\end{equation}

Our main tool in this section is the following.
\begin{prop}[label={prop:6}]{There is a set representing the same
    measures as a bounded weight}{}\rr
  Let $w$ be a bounded, sublacunary $R$-weight.

  Then there is a set $S\subset R$ so that
  \begin{equation}
    \label{eq:94}
    \lim_N\max_{\beta\in\setT}\abs*{\setA_{s\in
        S(N)}\e(s\beta)-\setA^{w}_{r\in R(N)}\e(r\beta)}=0
  \end{equation}
  As a consequence, if the $R$-weight $w$ is good then so is the set
  $S$ and we have
  \begin{equation}
    \label{eq:95}
    \mu_{S,\beta}=\mu_{w,\beta} \fe \beta
  \end{equation}
\end{prop}

\begin{proof}
  Since we can always assume that the bound of the $R$-weight $w$ is
  $1$, \cref{prop:6} follows from the following lemma.
  \begin{lem}[label={lem:4}]{Random selection of a good set}{}\rr
    Let $\sigma$ be an $R$-weight bounded by $1$.  We assume that for
    a constant $b>0$ we have
    \begin{equation}
      \label{eq:96}
      \liminf_N\frac{\sigma\paren[\big]{R(N)}}{\log N}> b
    \end{equation}

    Let $(\Omega,P)$ be a probability space and and let
    $(X_r)_{r\in R}$ be a sequence of totally independent
    $\Omega\to \cbrace{0,1}$ random variables indexed by $R$ and with
    distribution $P(X_r=1)=\sigma(r)$ (so $P(X_r=0)=1-\sigma(r)$).
 
    Then we have
    \begin{equation}
      \label{eq:97}
      P\set*{\omega}{\sup_N\max_{\beta\in\setT}
        \frac{\abs*{\sum_{r\in
              R(N)}\paren[\Big]{X_r(\omega)-\sigma(r)}\e(r\beta)}}
        {\sqrt{(\log N)\sigma\paren[\big]{R(N)}}}<\infty}=1
    \end{equation}
  \end{lem}
  To see that \cref{prop:6} indeed follows from \cref{lem:4}, let
  $\sigma=\frac{w}{\ninfty{w}}$, so $\sigma$ is bounded by $1$.  Here
  we make a bit more complicated argument than needed to show that
  there is a rate of convergence in \cref{eq:94}.

  The sublacunarity assumption on $w$ implies that $\sigma$ is
  sublacunary.  We then have, as a consequence of \cref{eq:97}, that
  there is a measurable subset $\Omega_1$ of $\Omega$ with
  $P(\Omega_1)=1$ so that for every $\omega\in \Omega_1$ there is a
  finite positive constant $C_\omega$ with
  \begin{equation}
    \label{eq:98}
    \max_{\beta\in\setT}
    \abs*{\frac1{\sigma\pa*{R(N)}}\sum_{r\in R(N)}X_r(\omega)\e(r\beta)
      -\frac1{\sigma\pa*{R(N)}}\sum_{r\in R(N)}\sigma(r)\e(r\beta)}
    \le C_\omega\sqrt{\frac{\log N}{\sigma\paren[\big]{R(N)}}}
  \end{equation}
  For $\beta=0$, we then have
  \begin{equation}
    \label{eq:99}
    \abs*{\frac1{\sigma\pa*{R(N)}}\sum_{r\in R(N)}X_r(\omega)-1}
    \le C_\omega\sqrt{\frac{\log N}{\sigma\paren[\big]{R(N)}}}
  \end{equation}
  This implies that if we replace $\sigma\pa*{R(N)}$ by
  $\sum_{r\in R(N)}X_r(\omega)$ in
  $\frac1{\sigma\pa*{R(N)}}\sum_{r\in R(N)}X_r(\omega)\e(r\beta)$ we
  make a $O\pa*{ \sqrt{\frac{\log N}{\sigma\paren[\big]{R(N)}}}}$
  error, hence \cref{eq:98} implies
  \begin{equation}
    \label{eq:100}
    \max_{\beta\in\setT}
    \abs*{\frac1{\sum_{r\in R(N)}X_r(\omega)}\sum_{r\in R(N)}X_r(\omega)\e(r\beta)
      -\frac1{\sigma\pa*{R(N)}}\sum_{r\in R(N)}\sigma(r)\e(r\beta)}
    \le C_\omega\sqrt{\frac{\log N}{\sigma\paren[\big]{R(N)}}}
  \end{equation}
  Defining $S_\omega\subset R$ by
  \begin{equation}
    \label{eq:101}
    S_\omega\cq \set*{r}{r\in R,X_r(\omega)=1}
  \end{equation}
  we can write \cref{eq:100} as
  \begin{equation}
    \label{eq:102}
    \max_{\beta\in\setT}
    \abs*{\setA_{s\in S_\omega(N)}\e(s\beta)
      -\setA^{\sigma}_{r\in R(N)}\e(r\beta)}
    \le C_\omega\sqrt{\frac{\log N}{\sigma\paren[\big]{R(N)}}}\fe \omega\in\Omega_1
  \end{equation}
  Since $\sigma$ is a constant multiple of $w$, we can replace
  $\sigma$ by $w$ in \cref{eq:102},
  \begin{equation}
    \label{eq:103}
    \max_{\beta\in\setT}
    \abs*{\setA_{s\in S_\omega(N)}\e(s\beta)
      -\setA^{w}_{r\in R(N)}\e(r\beta)}
    \le C_\omega\sqrt{\frac{\ninfty{w}\log N}{w\paren[\big]{R(N)}}}\fe \omega\in\Omega_1
  \end{equation}
  Since $\lim_N \frac{\ninfty{w}\log N}{w\paren[\big]{R(N)}}=0$, due
  to the sublacunarity assumption on the $R$-weight $w$, we get
  \cref{eq:94} if we take $S=S_\omega$ for any $\omega\in\Omega_1$.

  \begin{proof}[Proof of \cref{lem:4}]
    To see clearly what we need to do, denote
    \begin{align*}
      Z_N(\beta)&\coloneqq\sum_{r\in R(N)} \paren[\Big]{X_r(\omega)-\sigma(r)}\e(r\beta)\\
      \intertext{and}
      t_N&\coloneqq c\cdot\sqrt{(\log N)\sigma\paren[\big]{R(N)}}
    \end{align*}
    where we'll choose the constant $c$ appropriately later. By the
    Borel-Cantelli lemma, it's enough to prove
    \begin{equation}
      \sum_N P\paren*{\max_{\beta\in\setT}\abs*{Z_N(\beta)}\ge t_N}<\infty\label{eq:104}
    \end{equation}
    The first idea in proving \cref{eq:104} is that we do not have to
    take the maximum over all $\beta\in\setT$, but over a finite
    subset $B$ of $\setT$ which contains $N^3$ elements\footnote{In
      fact, we can take a set $B$ with as few elements as $10N$, but
      in our applications, $10N$ won't improve anything over $N^3$.}.
    Since the degree of the trigonometric polynomial $Z_N(\beta)$ is
    at most $N$, we can readily see that
    $\sup_{\beta\in\setT}|Z_N'(\beta)|\le
    N^2\sup_{\beta\in\setT}|Z_N(\beta)|$.  It follows that if we take
    $B_N\subset \setT$ to be an arithmetic progression with
    $|B_N|=N^3$ then
    \begin{equation}
      \label{eq:105}
      \max_{\beta\in\setT}\abs*{Z_N(\beta)}\le 2\max_{\beta\in B_N}\abs*{Z_N(\beta)}
    \end{equation}
    Hence we have
    \begin{equation}
      \label{eq:106}
      P\paren*{\max_{\beta\in\setT}\abs*{Z_N(\beta)}\ge t_N}\le
      P\paren*{\max_{\beta\in B_N}\abs*{Z_N(\beta)}\ge t_N/2}
    \end{equation}
    Using the union estimate, we get
    \begin{equation}
      \label{eq:107}
      P\pa*{\max_{\beta\in B_N}\abs*{Z_N(\beta)}\ge t_N/2}\le
      N^3\max_{\beta\in B_N}P\pa*{\abs*{Z_N(\beta)}\ge t_N/2}
    \end{equation}
    Thus \cref{eq:104} follows from
    \begin{equation}
      \label{eq:108}
      \sum_NN^3\max_{\beta\in B_N}P\paren*{\abs*{Z_N(\beta)}\ge t_N/2}<\infty
    \end{equation}
    This follows if we prove
    \begin{equation}
      \label{eq:109}
      P\paren[\Big]{\abs*{Z_N(\beta)}\ge t_N/2}<\frac2{N^5}
      \fe \beta\in\setT
    \end{equation}
    To prove \cref{eq:109}, we use the Bernstein-Chernoff exponential
    estimate\autocite[Exercise 1.3.4 with
    $t=\lambda\sigma$]{tao_vu_2006}.  This estimate says that if
    $Y_k$, $k\in[1,K]$, are totally independent, mean zero, complex
    valued random variables with $\abs*{Y_k}\le 1$, then
    \begin{equation}
      \label{eq:110}
      P\paren*{\abs*{\sum_{k\in[1,K]} Y_k}\ge t}
      \le
      4\max\cbrace*{\exp\paren*{-\frac{t^2/8}{\sum_{k\in[1,K]}\setE|Y_k|^2}},\exp\paren*{-t/3}} 
      \fe  t>0
    \end{equation}

    Take $K=\#R(N)$ and $Y_r(\beta)\coloneqq(X_r-\sigma(r))e(r\beta)$
    for $r\in R(N)$.  Then $ |Y_r(\beta)|\le 1$ so the $Y_r$ satisfy
    the assumption in Bernstein's inequality, hence, with $t=t_N/2$,
    we get the estimate
    \begin{equation}
      \label{eq:111}
      P\paren[\Big]{\abs*{Z_N(\beta)}\ge t_N/2}
      \le
      4\max\cbrace*{\exp\paren*{-\frac{t_N^2/32}{\sum_{r\in R(N)}\setE|Y_r|^2}},\exp\paren*{-t_N/6}}
    \end{equation}
    Since $\setE |Y_r(\beta)|^2= \sigma(r)(1-\sigma(r))$ we have
    \begin{equation}
      \label{eq:112}
      \sum_{r\in R(N)}\setE |Y_r(\beta)|^2
      \le \sigma\paren[\big]{R(N)}
    \end{equation}
    Using that $t_N=c\cdot\sqrt{(\log N)\sigma\paren[\big]{R(N)}}$, we
    get
    \begin{align*}
      \frac{t^2_N/32}{\sum_{r\in R(N)}\setE |Y_r(\beta)|^2}
      &=\frac{(c^2/32)(\log N)\sigma\paren[\big]{R(N)}}
        {\sum_{r\in R(N)}\setE |Y_r(\beta)|^2}\\
      \intertext{using the estimate in \cref{eq:112}} 
      &\ge \frac{(c^2/32)(\log N)\sigma\paren[\big]{R(N)}}
        {\sigma\paren[\big]{R(N)}}\\
      &=(c^2/32)(\log N)
    \end{align*}
    hence
    \begin{equation}
      \label{eq:113}
      \exp\paren*{-\frac{t^2_N/32}{\sum_{r\in R(N)}\setE|Y_r(\beta)|^2}}
      \le e^{-(c^2/32)(\log N)}
    \end{equation}
    In order to get $e^{-(c^2/32)(\log N)}\le N^{-5}=e^{-5\log N}$, we
    need to have $c^2/32 \ge 5$, so it enough to have, since
    $\sqrt{160}<13$,
    \begin{equation}
      \label{eq:114}
      c\ge 13
    \end{equation}
    We also have
    \begin{align*}
      t_N/6&=(c/6)\cdot\sqrt{(\log N)\sigma\paren[\big]{R(N)}}\\
      \intertext{by the assumption in \cref{eq:96} for all large enough $N$}
           &\ge (c/6)\sqrt b\log N
    \end{align*}
  
    It follows that
    \begin{equation}
      \label{eq:115}
      \exp(-t_N/6)\le e^{-(c/6)\sqrt b\log N}
    \end{equation}
    We again need to have
    $e^{-(c/6)\sqrt b\log N}\le N^{-5}=e^{-5\log N}$ which poses the
    requirement $(c/6)\sqrt b\ge 5$, that is,
    \begin{equation}
      \label{eq:116}
      c\ge \frac{30}{\sqrt b}
    \end{equation}
    Thus choosing the constant $c$ large enough to satisfy both
    \cref{eq:116,eq:114}, the estimate in \cref{eq:111} implies the
    one in \cref{eq:109}.
  \end{proof}
\end{proof}

\subsection{Notes to \cref{lem:4}}
\label{sec:notes-creflem:5}
The type of method we used in \cref{lem:4} to estimate trigonometric
polynomials goes back to Salem-Zygmund\autocite[Chapter IV]{MR65679}.
Recent developments have been given for example by
Weber\autocite{MR1768824} and by Cohen-Cuny\autocite{MR2206342}.

\section{Absolute continuity and positive mean}
\label{sec:absol-cont-posit}

The general theme of this section is that if a good set or weight has
positive mean then it can represent only an absolutely continuous
measure.  To be specific, we want to prove \cref{item:2,item:4}.

Our standing assumption is that $R$ is a sublacunary good set, and
hence we suppress it in our notation for the limit measure, so we
write $\mu_\alpha$ instead of $\mu_{R,\alpha}$.

\subsection{Proof of \cref{item:2}}
\label{sec:proof-cref-crefthm:p}

\Cref{item:1}  says that if $\rho$
is an unsigned $L^{\infty}\pa*{\mu_\alpha}$ function with
$\mu_{\alpha}(\rho)>0$ and $\alpha$ is an irrational number then
$\rho$ can be represented at $\alpha$ with a good set $S\subset R$
satisfying
$\M_R(S)=\frac{\mu_{\alpha}(\rho)}{\norm*{\rho}_{L^\infty\pa*{\mu_{\alpha}}}}$.
We have proved this in \cref{sec:proof-bounded-rho}.

\Cref{item:2} says that the
converse is also true: if the good set $S\subset R$ satisfies
$\noner{S}>0$ then the limit measure $\mu_{S,\beta}$ is absolutely
continuous with respect to $\mu_\beta$ with a bounded Radon-Nikodym
derivative $\rho_\beta$ which must satisfy
\begin{equation}
  \label{eq:117}
  \norm*{\rho_\beta}_{L^\infty\pa*{\mu_{\beta}}}\le \frac{1}{\noner{S}} \fe \beta
\end{equation}
This is what we intend to prove now.  Since $\beta\in\setT$ is fixed,
we suppress it in our notation, so for example we write $\mu$ for
$\mu_\beta$ and $\mu_S$ for $\mu_{S,\beta}$.  Let $S\subset R$ be such
that $\noner{S}>0$. Let us first show that for every $\beta$, the limit
measure $\mu_{S}$ is absolutely continuous with respect to $\mu$.

This will follow if we show that for every Borel set $B$ we have
\begin{equation}
  \label{eq:118}
  \mu_{S}(B)\le \frac1{\noner{S}} \mu(B)
\end{equation}
To see this, it's enough to show that for every unsigned, continuous
function $\phi$ on $\setT$ we have
\begin{equation}
  \label{eq:119}
  \mu_{S}(\phi)\le \frac1{\noner{S}} \mu(\phi)
\end{equation}
Let $\phi$ be such a function and let $N_1<N_2<\dots$ be a sequence of
indices for which $\lim_k\setA_{r\in R(N_k)}\setone_S(r)=\noner{S}$. We
can then estimate as
\begin{align*}
  \mu_{S}(\phi) & =\lim_N\setA_{s\in S(N)}\phi(s\beta)                                  \\
                & =\lim_k\setA_{s\in S(N_k)}\phi(s\beta)                                \\
                & =\lim_k\frac{1}{\setA_{r\in R(N_k)}\setone_S(r)}
                  \setA_{r\in R(N_k)}\setone_S(r)\phi(r\beta) \\
                &\le\limsup_k\frac{1}{\setA_{r\in R(N_k)}\setone_S(r)}
                  \setA_{r\in R(N_k)}\phi(r\beta) \\
  \intertext{ since
  $\lim_k\frac{1}{\setA_{n\in [1,N_k]}\setone_S(n)}=\frac1{\noner{S}}$ and
  $\lim_N\setA_{r\in R(N_k)}\phi(r\beta)$ exists,}
                & =\frac1{\noner{S}}\lim_N \setA_{r\in
                  R(N_k)}\phi(r\beta)       \\
  \intertext{since $\lim_N\setA_{r\in R(N_k)}\phi(r\beta)=\mu(\phi)$,}
                & =\frac1{\noner{S}} \cdot \mu(\phi)
\end{align*}
proving \cref{eq:119}.

Now, inequality
$\mu\pa*{\rho_{\beta}\setone_B} \le \frac1{\noner{S}} \mu(B)$ applied
to the Borel set $B=\cbrace*{\rho_{\beta} > \frac1{\noner{S}}}$ readily
gives \cref{eq:117}.

\subsection{Proof of \cref{item:4}}
\label{sec:proof-cref-crefthm:r}

Since the good set $R$ is fixed, we suppress it in our notation for
the limit measures, so we write $\mu_\alpha$ instead of
$\mu_{R,\alpha}$. 

In this section, we need to prove that if the good $R$-weight $w$ has
positive relative $\mathbf 1$-norm and it is integrable, that is, it
can be approximated arbitrary closely by bounded, good $R$-weights in
$\noner{}$-seminorm, then for every irrational $\beta$ the limit
measure $\mu_{w,\beta}$ is absolutely continuous with respect to
$\mu_\beta$.

Let $(w_k)$ be a sequence of good, bounded $R$-weights which converges
to $w$ in $\noner{}$-seminorm, $\lim_k\noner{w_k-w}=0$. Since
$\abs*{\noner{w_k}-\noner{w}}\le \noner{w_k-w}$, we have
$\lim_k\noner{w_k}=\noner{w}>0$, and hence we can assume without loss of
generality that $\noner{w_k}>0$ for every $k$.  That for every $k$ the
measure $\mu_{w_k,\beta}$ is absolutely continuous with respect to
$\mu_{\beta}$ for every $\beta$ follows from
\begin{equation}
  \label{eq:120}
  \mu_{w_k,\beta}(B)\le \frac{\ninfty{w_k}}{\noner{w_k}}\mu_{\beta}(B)
  \fe \text{ Borel set }B
\end{equation}
The proof of this inequality is almost identical to the proof of the
inequality in \cref{eq:118}, hence we omit it.

Now the rest of the proof of \cref{thm:representation_by_weights}
follows from \cref{lem:3}.

\section{Proof of \cref{thm:general_representability} for unbounded
  $\rho$}
\label{sec:proof-crefthm:g-unbo}
In this section we again work with a fixed, sublacunary good set
$R\subset \setN$ which we view as a sequence $(r_n)$ arranged in
increasing order. We omit $R$ from our notation for the limit
measures, so we write $\mu_\beta$ instead of $\mu_{R,\beta}$. We also
fix an irrational number $\alpha$.  Let $\rho\in L^1\pa*{\mu_\alpha}$.
We want to find a good set\footnote{Which can be shown to be
  sublacunary as a consequence of the sublacunarity of the weight $v$
  below.} $S\subset R$ which represents $\rho$ at $\alpha$.  According
to \cref{prop:4} there is a good $R$-weight $w$ which represents
$\rho$ at $\alpha$.  Since this weight $w$ has positive relative mean
with respect to $R$, it's a sublacunary weight.  The problem is that,
as per construction, $w$ is not a bounded weight if $\rho$ is
unbounded, hence we cannot use our \cref{prop:6} to construct the
desired set $S$.

Our main job in this section hence will be to construct a good
$R$-weight $v$ satisfying the following properties
\begin{itemize}
\item $v$ is bounded by $1$;
  
\item $v$ is sublacunary;
  
\item $v$ represents the same measure at every $\beta$ as $w$, so
  $\mu_{v,\beta}= \mu_{w,\beta}$ for every $\beta$.
\end{itemize}
Once we have such a good weight $v$, we can use \cref{prop:6} to
``construct'' the desired good set $S$.

The weight $v$ will be of the form $\sigma\cdot w$ where the weight
$\sigma$ is a \emph{decreasing} weight, that is,
$\sigma(r_n)\ge \sigma(r_{n+1})$ for every $n\in\setN$.  That a weight
$v$ of this form represents the same measures everywhere is a
consequence of a general but probably familiar result---our main new
tool in this section.  Not to get bugged down with unnecessary
notation, we will state the result for weights with the reindexing
$w(n)=w(r_n)$ with which $R$ weights become $\setN$-weights.

First recall the definition of a dissipative sequence of measures on
$\setN$.
\begin{defn}[label={defn:8}]{Dissipative sequence of measures}{}\rr
  Let $(v_N)_{N\in\setN}$ be a sequence of finite measures on $\setN$.

  We say, the sequence $(v_N)_{N\in\setN}$ is \emph{dissipative} if
  \begin{equation}
    \label{eq:121}
    \lim_N\frac{v_N(j)}{v_N(\setN)}=0, \text{ for every }j\in\setN
  \end{equation}
\end{defn}

\begin{prop}[label={prop:7}]{Decreasing weights preserve limits}{}\rr  
  Let $w$ be a weight, $(\sigma_N)_{N\in\setN}$ be a sequence of
  finite measures on $\setN$ and let $\x=(x_n)$ be a sequence from a
  normed space $(X,\norm{})$.  Denoting $v_N\cq \sigma_N\cdot w$, we
  assume the following
  \begin{propenum}
  \item Each $\sigma_N$ has finite support.
  \item The sequence $(v_N)$ is dissipative.
  \item For each $N$ the measure $\sigma_N$ is decreasing,
    $\sigma_N(1)\ge \sigma_N(2)\ge \dots$.
  \item The sequence $\pa*{\setA_{n\in[1,N]}^{w}x_n}_N$ converges to
    some $y\in X$,
    \begin{equation}
      \label{eq:122}
      \lim_N\setA_{n\in[1,N]}^{w}x_n=y
    \end{equation}
  \end{propenum}
   
  Then, the sequence $\pa*{\setA^{v_N}_{j\in\setN}x_j}_N$ of averages
  converge to the same limit as the $w$-weighted averages,
  \begin{equation}
    \label{eq:123}
    \lim_N\setA^{v_N}_{j\in\setN}x_j=y
  \end{equation}
  At the heart of this result is the following quantitative estimate:
  For a given $\epsilon>0$, if $K$ is such that
  $\norm*{\setA_{n\in[1,j]}^{w}x_n-y}<\epsilon$ for $j\ge K$ then we
  have
  \begin{equation}
    \label{eq:124}
    \norm*{\setA^{v_N}_{j\in\setN}x_j-y}
    \le \epsilon+\max_{j\in [1,K]}\norm*{\setA_{n\in[1,j]}^{w}x_n-y}
    \cdot \frac{v_N([1,K])}{v_N(\setN)}
  \end{equation}
  for every $N\ge K$.
\end{prop}
Note that the estimate in \cref{eq:124} indeed implies the conclusion
of the proposition in \cref{eq:123}.  To see this, let $N\to\infty$ in
\cref{eq:124}. Then, since $(v_N)$ is a dissipative sequence so
$\lim_N\frac{v_N([1,K])}{v_N(\setN)}=0$, we get that
$\limsup_N\norm*{\setA^{v_N}_{j\in\setN}x_j-y}\le\epsilon$.  Since
$\epsilon>0$ is arbitrary, we get
$\lim_N\norm*{\setA^{v_N}_{j\in\setN}x_j-y}=0$.

\begin{proof}[Proof of \cref{prop:7}]
  The main idea of the proof is to write $\setA^{v_N}_{j\in\setN}x_j$
  as an average of the $w$-averages with respect to another measure
  $q_N$ on $\setN$
  \begin{equation}
    \label{eq:125}
    \setA^{v_N}_{j\in\setN}x_j=\setA^{q_N}_{j\in\setN}\setA^{w}_{n\in[1,j]}x_n
    \text{ for all }N
  \end{equation}
  These measures $q_N$ will also satisfy
  \begin{equation}
    \label{eq:126}
    q_N(\setN)=v_N(\setN)\text{ for every }N\in\setN
  \end{equation}
  The measure $q_N$ appears during performing summation by parts:
  setting $\sigma_N(0)\coloneqq 0$, $w(0)\cq 0$ and $x_0\coloneqq 0$,
  we have
  \begin{align*}
    \setA^{v_N}_{j\in\setN}x_j
    &=\frac1{v_N(\setN)}\sum_{j\in\setN}\sigma_N(j)w(j)x_j\\
    &=\frac1{v_N(\setN)}
      \sum_{j\in\setN}\sigma_N(j)
      \paren*{\sum_{n\in[1,j]}w(n)x_n-\sum_{n\in[j-1]}w(n)x_n}\\
    &=\frac1{v_N(\setN)}
      \sum_{j\in\setN}\paren[\Big]{\sigma_N(j)-\sigma_N(j+1)}
      \sum_{n\in[1,j]}w(n)x_n\\
    &=\frac1{v_N(\setN)}
      \sum_{j\in\setN}\paren[\Big]{\sigma_N(j)-\sigma_N(j+1)}
      \cdot w([1,j]) \cdot \setA^{w}_{n\in[1,j]}x_n
  \end{align*}
  Thus, defining the measure $q_N$ by
  \begin{equation}
    \label{eq:127}
    q_{N}(j)\coloneqq \paren[\Big]{\sigma_N(j)-\sigma_N(j+1)}\cdot w([1,j]), \text{
      for }j\in\setN
  \end{equation}
  we get the identity in \cref{eq:125} once we show that $q_N$ really
  is a measure satisfying \cref{eq:126}. That $q_N(j)$ is unsigned
  follows from the assumption that the sequence
  $(\sigma_N(j))_{j\in\setN}$ is decreasing for fixed $N$. That
  $q_N(\setN)=v_N(\setN)$ follows by setting $x_j=1$ for every $j$ in
  the summation by parts argument above since then we get exactly
  $q_N(\setN)=v_N(\setN)$:
  \begin{align*}
    1&=\setA^{v_N}_{j\in\setN}1\\
     &=\frac1{v_N(\setN)}
       \sum_{j\in\setN}\paren[\Big]{\sigma_N(j)-\sigma_N(j+1)}
       \cdot w([1,j]) \cdot \setA^{w}_{n\in[1,j]}1\\
     &=\frac1{v_N(\setN)}\sum_{j\in\setN}q_N(j)\cdot1\\
     &=\frac1{v_N(\setN)} \cdot q_N(\setN)
  \end{align*}

  Using the now obvious identity $y=\setA^{q_N}_{j\in\setN}y$ together
  with \cref{eq:125}, we can now write $\setA^{v_N}_{j\in\setN}x_j-y$
  as
  \begin{equation}
    \label{eq:128}
    \setA^{v_N}_{j\in\setN}x_j-y
    =\setA^{q_N}_{j\in\setN}\paren*{\setA^{w}_{n\in[1,j]}x_n-y}
  \end{equation}

  Let $\epsilon>0$.  Since we assumed
  $\lim_N \setA^{w}_{n\in[1,N]}x_n=y$, there is an $K=K(\epsilon)$ so
  that
  \begin{equation}
    \label{eq:129}
    \norm*{\setA^w_{n\in[1,j]}x_n-y}<\epsilon, \text{ for }j\ge K
  \end{equation}
  Splitting the summation on $j$ in
  $\setA^{q_N}_{j\in\setN}\paren*{\setA^w_{n\in[1,j]}x_n-y}$ into two
  parts at $K$ and using the triangle inequality, we get the estimate
  \begin{multline}
    \label{eq:130}
    \norm*{\setA^{q_N}_{j\in\setN}\paren*{\setA^w_{n\in[1,j]}x_n-y}}\le
    \norm*{\frac1{q_N(\setN)}\sum_{j\in[1,K]}q_N(j)\paren*{\setA^w_{n\in[1,j]}x_n-y}}\\
    +\norm*{\frac1{q_N(\setN)}\sum_{j> K}q_{N}(j)\paren*{\setA^w_{n\in[1,j]}x_n-y}}
  \end{multline}

  We can estimate the first term as
  \begin{equation}
    \label{eq:131}
    \norm*{\frac1{q_N(\setN)}\sum_{j\in[1,K]}q_{N}(j)\paren*{\setA^w_{n\in[1,j]}x_n-y}}
    \le \max_{j\in[1,K]}\norm*{\setA^w_{n\in[1,j]}x_n-y}\cdot \frac{q_{N}([1,K])}{q_N(\setN)}
  \end{equation}
  Using the definition of $q_N(j)$ as given in \cref{eq:127}, we can
  estimate $q_{N}([1,K])$ as
  \begin{align*}
    q_{N}([1,K])
    &=\sum_{j\in[1,K]}\pa[\Big]{\sigma_N(j)-\sigma_N(j+1)}\cdot
      w([1,j])\\
    &=\sum_{j\in[1,K]}\sigma_N(j)\pa[\Big]{w([1,j])-w([j-1])}
      -\sigma_N(K+1)w([1,K])\\
    &=\sum_{j\in[1,K]}\sigma_N(j)w(j) -\sigma_N(K+1)w([1,K])\\
    &=\sum_{j\in[1,K]}v_N(j)-\sigma_N(K+1)w([1,K])\\
    &\le v_N([1,K])
  \end{align*}
  Using this estimate and that $q_{N}(\setN)=v_{N}(\setN)$ in
  \cref{eq:131} we get
  \begin{equation}
    \label{eq:132}
    \norm*{\frac{1}{q_N(\setN)}
      \sum_{j\in[1,K]}q_{N}(j)\paren*{\setA^w_{n\in[1,j]}x_n-y}}
    \le \max_{j\in[1,K]}\norm*{\setA^w_{n\in[1,j]}x_n-y}\cdot
    \frac{v_{N}([1,K])}{v_N(\setN)}
  \end{equation}

  The second term in \cref{eq:130} can be estimated, using
  \cref{eq:129}, as
  \begin{equation}
    \label{eq:133}
    \norm*{\frac{1}{q_N(\setN)}
      \sum_{j> K}q_{N}(j)\paren*{\setA^w_{n\in[1,j]}x_n-y}}\le \epsilon
  \end{equation}
  Putting the estimates in \cref{eq:132,eq:133} into \cref{eq:130} and
  using the identity in \cref{eq:128} we get \cref{eq:124}.
\end{proof}

\begin{cor}[label={cor:1}]{Decreasing weights preserve limit measures
    of weights}{}\rr 
  Let $w$ and $\sigma$ be $R$-weights.  Denoting $v\cq \sigma\cdot w$,
  we assume the following

  \begin{corenum}
  \item $v(R)=\infty$.
  \item The $R$-weight $\sigma$ is decreasing
    $\sigma(r_1)\ge \sigma(r_2)\ge \dots$.
  \item The $R$-weight $w$ is good.
  \end{corenum}

  Then $v$ is a good $R$-weight and it represents the same measures
  everywhere as $w$,
  \begin{equation}
    \label{eq:134}
    \mu_{v,\beta}=\mu_{w,\beta}\fe \beta
  \end{equation}
\end{cor}
\begin{proof}
  We need to show that for a given $\beta$ we have
  \begin{equation}
    \label{eq:135}
    \lim_N\setA^v_{n\in[1,N]}\e(r_n\beta)=\mu_{w,\beta}(\e)
  \end{equation}
  to do this, use \cref{prop:7} with $\sigma_N$ defined by
  \begin{equation}
    \label{eq:136}
    \sigma_N(n)\cq \sigma(r_n)\setone_{[1,N]}(n)
  \end{equation}
  and $(x_n)$ defined by
  \begin{equation}
    \label{eq:137}
    x_n\cq \e(r_n\beta)
  \end{equation}
\end{proof}
Let us now go back to our good $R$-weight $w$ which represents $\rho$
at $\alpha$.  Since we now consider $R$ as the sequence $(r_n)$, its
sublacunarity assumption is expressed more conveniently as
\begin{equation}
  \label{eq:138}
  \lim_N\frac{N}{\log r_N}=\infty
\end{equation}
as we noted in \cref{eq:11}.  Since the weight $w$ satisfies
$\M_R(w)>0$, \cref{eq:138} implies that $w$ is also sublacunary.
Writing $\frac{N+1}{\log r_{N+1}}=\frac{N+1}{N}\cdot \frac{N}{\log
  r_{N+1}}$ we see that \cref{eq:138} implies
\begin{equation}
  \label{eq:139}
  \lim_N\frac{N}{\log r_{N+1}}=\infty
\end{equation}
According to the proof of \cref{lem:1}, we obtained $w$ as the limit
of a sequence $(w_k)$ of bounded good weights by pasting the $w_k$
together piece by piece in a sense that after choosing indices
$N_1<N_2<\dots$, we define $w$ to be equal $w_k$ on the interval
$(N_k,N_{k+1}]$
\begin{equation}
  \label{eq:140}
  w(r_n)\cq \sum_kw_k(r_n)\setone_{(N_k,N_{k+1}]}(n)
\end{equation}
Now, in order to obtain a good weight $v$ which is bounded by $1$ and
would represent the same measures as $w$, we could do the following.
Define $\sigma$ by
\begin{equation}
  \label{eq:141}
  \sigma(r_n)\cq
  \frac1{\max_{j\in [1,k]}\ninfty{w_j}}\cdot \setone_{(N_k,N_{k+1}]}(n)
\end{equation}
Then $\sigma$ is decreasing and $v\cq \sigma w$ is bounded by $1$.
The remaining issue is to ensure that $v$ is sublacunary, and to do
that it's enough to ensure 
\begin{equation}
  \label{eq:142}
  \lim_N\frac{\sum_{n\in[1,N]}v(r_n)}{\log r_{N+1}}=\infty
\end{equation}
as we noted in \cref{eq:93}.  This would also ensure that both
$\sigma$ and $v$ are weights.  It turns out that in the recursive
process of choosing the indices $(N_k)$ if we choose $N_k$ large
enough compared to $N_{k-1}$ we can ensure that $v$ is sublacunary.
We want to show that we can choose the indices $N_k$ so that we will
have \cref{eq:142}.  Let us note that in the proof of \cref{lem:1} the
choice of $N_k$ is flexible, since it just has to be large enough to
staisfy some criteria.  So we now add one additional criterion, namely
we want to choose $N_k$ large enough to also satisfy
\begin{equation}
  \label{eq:143}
  \frac{N}{\max_{j\in[1,k]}\ninfty{w_j}} >k \log r_{N+1} \fe N\ge N_k
\end{equation}
This is possible because of the sublacunarity condition in
\cref{eq:139}, and \cref{eq:143} ensures the sublacunarity of $v$,
that is, \cref{eq:142}.
  
That $v$ represents the same measures as $w$ at every $\beta$ follows
from \cref{cor:1}. As in the last step of our proof of
\cref{thm:general_representability}, we use \cref{prop:6} to show the
existence of a good set $S\subset R$ which represents the same
measures as $v$ at every $\beta$, hence at $\beta=\alpha$ we have
$\mu_{S,\alpha}=\rho\mu_\alpha$.

\section{The limit measure at rational points}
\label{sec:limit-measure-at}
In this section we want to prove \cref{thm:representation_rational}.
The base set is $\setN$ which we suppress in our notation, so we write
$\mu_\beta$ instead of $\mu_{\setN,\beta}$.

Given the probability measure $\nu$ on $\setT_q$ and the rational
number $\frac{a}{q}$, $\gcd(a,q)=1$, let us see what properties a good
set $S$ would need to have so that $\mu_{S,a/q}=\nu$.

Introducing the sets $S_j$ by
\begin{equation}
  \label{eq:144}
  S_j\cq \set*{s}{s\in S,sa\equiv j\pmod q}, \fe j\in[1,q]
\end{equation}
let us write, using that the $S_j$ are pairwise disjoint,
\begin{align*}
  \setA_{s\in S(N)}
  &=\frac{1}{\# S(N)}\sum_{s\in S(N)}\delta_{sa/q}\\
  &=\frac{1}{\# S(N)}\sum_{j\in[1,q]}\sum_{s\in S_j(N)}\delta_{j/q}\\
  &= \sum_{j\in[1,q]}\frac{\# S_j(N)}{\# S(N)}\delta_{j/q}
\end{align*}
If we make the assumption\footnote{In fact, the existence of
  $\lim_N\frac{\# S_j(N)}{\# S(N)}$ follows from $S$ being a good
  set.}  that $\lim_N\frac{\# S_j(N)}{\# S(N)}$ exists for every $j$
then, letting $N\to\infty$, we get
\begin{equation}
  \label{eq:145}
  \mu_{S,a/q}=\sum_{j\in[1,q]}\delta_{j/q}\lim_N\frac{\# S_j(N)}{\# S(N)}
\end{equation}
Since $\mu_{S,a/q}$ is supposed to be equal $\nu$, we get
\begin{equation}
  \label{eq:146}
  \lim_N\frac{\# S_j(N)}{\# S(N)} =\nu(j/q)
\end{equation}
This gives us the idea how to construct $S$: we start out from the set
$R_j$ defined by
\begin{equation}
  \label{eq:147}
  R_j\cq \set*{n}{na\equiv j\pmod q}, \fe j\in[1,q]
\end{equation}
Note that $R_j$ is a full residue class $\mod q$, namely, if $j'$
denotes the unique solution to the congruence $j'a\equiv j\pmod q$,
then $R_j$ is the arithmetic progression $\set*{kq+j'}{k\in\setN}$.
Note that $R_j$ is a good set, as are all arithmetic progressions.  We
clearly have
\begin{equation}
  \label{eq:148}
  \M(R_j)=\frac1q \fe j\in[1,q]
\end{equation}
Now what remains is to find a set $S_j\subset R_j$ with relative mean
$\nu\pa*{\frac{j}{q}}$ and make sure that $S_j$ is a good set.  Let
$\gamma$ be an irrational number and consider
\begin{equation}
  \label{eq:149}
  S_j\cq \set*{r}{r\in R_j, r\gamma\in
    \left[0,\nu\pa*{\frac{j}{q}}\right)} \fe j\in[1,q]
\end{equation}
Using \cref{prop:1} with $\alpha=\gamma$ and $R=R_j$, we deduce that
$S_j$ is a good set with $\M_{R_j}(S_j)=\nu\pa*{\frac{j}{q}}$, as
desired.  We finally define $S$ as
\begin{equation}
  \label{eq:150}
  S\cq \bigcup_{j\in[1,q]} S_j
\end{equation}
The set $S$ is good since it's the finite union of pairwise disjoint
good sets with mean.  Indeed, we have
$\M(S_j)=\frac1q \cdot \nu\pa*{\frac{j}{q}}$ and hence
$\M(S)=\frac1q$.

\section{Examples}
\label{sec:examples}

\subsection{Two good sets, but their intersection has no mean.}
\label{sec:two-good-sets}

Here we construct randomly two good sets, $R,S$ with $\M(R)=\M(S)=1/2$
but $\M(R\cap S)$ doesn't exist.

Let $(X_n)$ be a iid sequence of random variables on the probability
space $(\Omega,P)$, modeling fair coin flipping, so with distribution
$P(X_n=1)=P(X_n=0)=1/2$.  Let us also consider another sequence of
random variables $(Y_n)$ defined by
\begin{equation}
  \label{eq:151}
  Y_n=
  \begin{cases}
    X_n&\text{ if } n\in [2^k,2^{k+1})\text{ for even }k\\
    1-X_n&\text{ if } n\in [2^k, 2^{k+1})\text{ for odd }k
  \end{cases}
\end{equation}
The $(Y_n)$ is also an iid sequence with the same distribution as the
$(X_n)$.  Define the sets $R^\omega,S^\omega$ by
$R^\omega\cq\set*{n}{X_n(\omega)=1}$ and
$S^\omega\cq\set*{n}{Y_n(\omega)=1}$. By \cref{lem:4} both $R^\omega$
and $S^\omega$ are good sets almost surely with
$\M(R^\omega)=\M(S^\omega)=1/2$. We claim that
$\M(R^\omega\cap S^\omega)$ almost surely doesn't exists.  To see
this, denote $T^\omega\cq R^\omega\cap S^\omega$ and observe that if
$\M(T^\omega)$ existed then
$\lim_k\frac{T^\omega\cap [2^k,2^{k+1})}{2^k}$ would exist.  But,
denoting by $O$ the odd numbers and by $E$ the even numbers, we almost
surely have
\begin{align*}
  \lim_{k\in O} \frac{T^\omega\cap [2^k,2^{k+1})}{2^k}&=0\\
  \lim_{k\in E} \frac{T^\omega\cap [2^k,2^{k+1})}{2^k}&=\frac12
\end{align*}

\subsection{$R_1\cup R_2$ and $R_1\cap R_2$ have means but are not
  good}
\label{sec:r_1cup-r_2-r_1cap}

Here is an example of two good sets $R_1$ and $R_2$ each with mean
$2/3$, $\M(R_1\cap R_2)=1/2$ but $R_1\cap R_2$ is not good and
$\M(R_1\cup R_2)=5/6$ but $R_1\cup R_2$ is not good.

Both sets will be defined in blocks of intervals.  Partition $\setN$
into a sequence of disjoint intervals $I_n$ so that their lengths go
to infinity but slower than the left endpoints go to infinity.  For
example, $I_n=[n^2,(n+1)^2)$ will do.

The first good set $R_1$ will contain all i\textbf{\emph{N}}tegers
from $I_1$, then only \textbf{\emph{O}}dd numbers from $I_2$ then
\textbf{\emph{E}}ven numbers from $I_3$ then repeat this pattern for
$I_4,I_5,I_6$ etc:

\begin{equation}
  \label{eq:152}
  N O E N O E\dots
\end{equation}

The set $R_2$ is defined similarly, except it will have one pattern in
intervals $J_k\coloneqq [3^k,3^{k+1})$ for even $k$ and another for
odd $k$.
\begin{equation}
  \label{eq:153}
  E O N E O N \dots\text{ for even }k
\end{equation}
\begin{equation}
  \label{eq:154}
  O N E O N E\dots\text{ for odd }k
\end{equation}

Both of these sets are good and they represent the same (uniform)
measure at every $\beta$.

The intersection $R_1\cap R_2$ has the patterns
\begin{equation}
  \label{eq:155}
  E O E E O E\dots\text{ for even }k
\end{equation}
\begin{equation}
  \label{eq:156}
  O O E O O E \dots\text{ for odd }k
\end{equation}

Clearly $\M(R_1\cap R_2)=1/2$ but the average of $\e(n/2)$ is
different on $J_k$ for even $k$ from those on odd $k$: for even $k$
the average will go to $1/3$ while for odd $k$ it goes to $-1/3$.

As for the union $R_1\cup R_2$, it has the patterns
\begin{equation}
  \label{eq:157}
  N O N N O N\dots\text{ for even }k
\end{equation}
\begin{equation}
  \label{eq:158}
  N N E N N E\dots\text{ for odd }k
\end{equation}
Clearly $\M(R_1\cup R_2)=5/6$ but the average of $\e(n/2)$ is
different on $J_k$ for even $k$ from those on odd $k$: for even $k$
the average will go to $-1/3$ while for odd $k$ it goes to $1/3$.

\subsection{Open set $U$ with visit set $\set{n}{n\alpha \in U}$ not
  good}
\label{sec:open-set-u}

Let $\alpha$ be an irrational number in the torus $\setT$. We
show that there exists an open subset $U$ of the torus such that the
sequence $\pa*{\setA_{n\in [1,N]} \setone_U(n\alpha)}_N$ does not
converge when $N$ goes to infinity.

\marginnote[-1cm]{The construction does not use at all the group
  structure or the dimensional properties of the torus. This can be
  extended in a general context of a sequence in a compact metric
  space with a non purely atomic asymptotic distribution.}

We want to construct an open subset $U$ of the torus and an increasing
sequence of positive integers $(N_k)_{k\geq0}$ such that the averages
$\setA_{n\in \sbrack*{1,N_{2k}}} \setone_U(n\alpha)$, $k=0,1,2,\dots$,
with even indices are large whereas the averages
$\setA_{n\in \sbrack*{1,N_{2k+1}}} \setone_U(n\alpha)$, $k=0,1,2,\dots$
with odd indices are small.

The sequence $(N_k)$ will be constructed by induction and each $N_k$
will be associated to $\epsilon_k\cq 1/(2^{k+4}N_k)$. In this
induction process, we construct also a sequence of open subsets
$(U_k)_{k\geq0}$.

We start with $N_0>1$ fixed and we define
\begin{equation*}
  U_0\cq\bigcup_{n\in [1,N_0]}\pa*{n\alpha-\epsilon_0,n\alpha+\epsilon_0}
\end{equation*}
We have of course
\begin{equation*}
  \setA_{n\in \sbrack*{1,N_{0}}} \setone_{U_0}(n\alpha)=1 \quad\text{and}\quad 0<\lambda\pa*{\overbar{U_0}}\leq 2N_0\epsilon_0
\end{equation*}

This is the initial step of our construction. In order to be
understandable, let us describe the two next steps.

By the uniform distribution of the sequence $(n\alpha)_n$ in the
torus, there exists a number $N_1>N_0$ such that
\begin{equation*}
  \setA_{n\in \sbrack*{1,N_{1}}}
  \setone_{\overbar{U_0}}(n\alpha)\le 2\lambda(\overbar{U_0})
  \le 4(N_0\epsilon_0)
\end{equation*}
We fix such a $N_1$. To any $n\in [1,N_1]$ with
$n\alpha\notin \overbar{U_0}$ we associate a real $\delta_n$ that
\begin{equation*}
  0<\delta_n\leq\epsilon_1\quad\text{and}\quad
  \pa*{n\alpha-\delta_n,n\alpha+\delta_n} \cap U_0=\emptyset
\end{equation*}
We define
\begin{equation*}
  U_1\cq\bigcup_{\substack{n\in [1,N_1]\\ n\alpha\notin \overbar{U_0}}}
  \pa*{n\alpha-\delta_n,n\alpha+\delta_n}
\end{equation*}
We have
\begin{equation*}
  \setA_{n\in \sbrack*{1,N_{1}}}
  \setone_{U_1}(n\alpha)\geq 1-4N_0\epsilon_0\quad\text{and}\quad
  0<\lambda\pa*{\overbar{U_1}}\leq 2N_1\epsilon_1
\end{equation*}

Note also that by construction $U_0\cap U_1=\emptyset$.

By the uniform distribution of the sequence $(n\alpha)_n$ in the
torus, there exists a number $N_2>N_1$ such that
\begin{equation*}
  \setA_{n\in \sbrack*{1,N_{2}}}
  \setone_{\overbar{U_1}}(n\alpha)\le
  2\lambda\pa*{ \overbar{U_1}}\leq 4(N_1\epsilon_1)
\end{equation*}
We fix such a $N_2$. To any $n\in [1,N_2]$ with
$n\alpha\notin \overbar{U_1}$ we associate a real $\delta_n$
satisfying
\begin{equation*}
  0<\delta_n\leq\epsilon_2\quad\text{and}\quad
  \pa*{n\alpha-\delta_n,n\alpha+\delta_n} \cap U_1=\emptyset
\end{equation*}
\marginnote{Note that the values of the $\delta_n$'s are
  reinitialized.}

We define
\begin{equation*}
  U_2\cq U_0\cup \bigcup_{\substack{n\in [1,N_2]\\ n\alpha\notin \overbar{U_1}}}
  \pa*{n\alpha-\delta_n,n\alpha+\delta_n}
\end{equation*}
We have
\begin{equation*}
  \setA_{n\in \sbrack*{1,N_{2}}}
  \setone_{U_2}(n\alpha)\geq 1-4N_1\epsilon_1
  \quad\text{and}\quad \lambda\pa*{\overbar{U_2}}\leq
  2N_0\epsilon_0+2N_2\epsilon_2
\end{equation*}
Note also that by construction $U_2\cap U_1=\emptyset$ and $U_0\subset U_2$.\\

Let us state now our induction hypothesis. Suppose that, for a fixed
integer $k>0$ we have already constructed two sequences
$$
\left(U_\ell\right)_{0\leq\ell\leq k}\quad\text{and}\quad N_0< N_1<
N_2<\ldots<N_k
$$
such that
\begin{itemize}
\item $ U_0\subset U_2 \subset U_4\subset\ldots$ and
  $ U_1\subset U_3\subset U_5\subset\ldots$,
\item If $\ell$ is even and $\ell'$ is odd, then $U_\ell$ and
  $U_{\ell'}$ are disjoint,
\item Each $U_{\ell}$ is a finite union of open intervals,
\item If $0\leq2\ell\leq k$, then
  \begin{align*}
    \lambda\pa*{U_{2\ell}}
    &\le 2\pa*{N_0\epsilon_0+N_2\epsilon_2+\ldots+N_{2\ell}\epsilon_{2\ell}}\\
    \intertext{and}
    \setA_{n\in \sbrack*{1,N_{2\ell}}}
    \setone_{U_{2\ell}}(n\alpha) &\geq
                                   1-4\pa*{N_1\epsilon_1+N_3\epsilon_3+\ldots+N_{2\ell-1}\epsilon_{2\ell-1}}
  \end{align*}
\item If $1\leq2\ell+1\leq k$, then
  \begin{align*}
    \lambda\pa*{U_{2\ell+1}}
    &\leq2\pa*{N_1\epsilon_1+N_3\epsilon_3+\ldots+N_{2\ell+1}\epsilon_{2\ell+1}}\\
    \intertext{and}
    \setA_{n\in \sbrack*{1,N_{2\ell+1}}}
    \setone_{U_{2\ell+1}}(n\alpha)
    &\geq 1-4\pa*{N_0\epsilon_0+N_2\epsilon_2+\ldots+N_{2\ell}\epsilon_{2\ell}}
  \end{align*}
\end{itemize}

Here begins the induction process. By the uniform distribution of the
sequence $(n\alpha)_n$ in the torus, there exists a number
$N_{k+1}>N_{k}$ such that

\begin{equation*}
  \setA_{n\in \sbrack*{1,N_{k+1}}}
  \setone_{\overbar{U_{k}}}(n\alpha)
  \leq 2\lambda\pa*{\overbar{U_{k}}}
\end{equation*}
We fix such a $N_{k+1}$. To any $n\in[1,N_{k+1}]$ with
$n\alpha\notin \overbar{U_{k}}$ we associate a real $\delta_n$ that

\begin{equation*}
  0<\delta_n\leq\epsilon_{k+1}\quad\text{and}\quad
  \pa*{n\alpha-\delta_n,n\alpha+\delta_n} \cap U_{k}=\emptyset
\end{equation*}
\marginnote{Note that the values of $\delta_n$'s are reinitialized at
  each induction step.}  We define
\begin{equation*}
  U_{k+1}\cq U_{k-1}\cup
  \bigcup_{\substack{n\in[1,N_{k+1}]\\ n\alpha\notin
    \overbar{U_{k}}}}
\pa*{n\alpha-\delta_n,n\alpha+\delta_n}
\end{equation*}

The items of the induction hypothesis are now satisfied by the
sequences $ \left(U_\ell\right)_{0\leq\ell\leq k+1}$\ and
$\left(N_\ell\right)_{0\leq\ell\leq k+1}$.

We can consider these sequences as infinite, and we define
$U\cq \bigcup_{k\geq0} U_{2k}$.

Recalling our choice $N_k\epsilon_k=2^{-k-4}$, we obtain
\begin{align*}
  \setA_{n\in \sbrack*{1,N_{2k}}}\setone_{U}(n\alpha)
  &\ge \setA_{n\in \sbrack*{1,N_{2k}}}\setone_{U_{2k}}(n\alpha)\\
  &\geq 1-4\sum_{\ell}N_{2\ell+1}\epsilon_{2\ell+1}\\
  &=5/6
\end{align*}
and
\begin{align*}
  \setA_{n\in \sbrack*{1,N_{2k+1}}}\setone_{U}(n\alpha)
  &\le \setA_{n\in\sbrack*{1,N_{2k+1}}}\setone_{U_{2k+1}^c}(n\alpha)\\
  &\le 4\sum_{\ell}N_{2\ell}\epsilon^{2\ell}\\
  &=1/3
\end{align*}


\printbibliography[heading=bibintoc]

\end{document}